\documentclass[10pt]{amsart}

\usepackage{amssymb}
\usepackage{bm}
\usepackage{graphicx}
\usepackage[centertags]{amsmath}
\usepackage{amsfonts}
\usepackage{amsthm}

\newtheorem{theorem}{Theorem}
\newtheorem*{mtheo}{Theorem}
\newtheorem{coroll}{Corollary}
\newtheorem{lemma}{Lemma}
\newtheorem{proposition}{Proposition}
\newtheorem{definition}{Definition}
\theoremstyle{definition}
\newtheorem{rem}{Remark}

\begin{document}
\title[ A canonical theorem for uniform families of strong subtrees]{A Canonical partition theorem for uniform families of finite strong subtrees}
%\runtitle{Partition theorem of finite strong subtrees}
\author{VLITAS Dimitris}
\footnotetext{The research leading to these results has received funding from the [European Community's] Seventh Framework Programme [FP7/2007-2013] under grant agreement n 238381}
 %------------------------Abstract-------------------------------%
\begin{abstract} Extending a result of K. Milliken \cite{Mi2}, in this paper we prove a Ramsey classification result for equivalence relations defined on uniform families of finite strong subtrees of a finite sequence $(U_i)_{i\in d}$ of fixed trees $U_i$, $i\in d$, that have a finite uniform branching but are of infinite length.

\end{abstract}
\maketitle

\section{Introduction}

Canonical results in Ramsey theory try to  describe equivalence relations in a given Ramsey structure,
based on the underlying pigeonhole principles. The first example of them is the classical Canonization
Theorem  by P. Erd\H{o}s and R. Rado \cite{Er-Ra} which can be presented as follows: Given $\alpha\le \beta
\le \omega$ let
$$\binom{\beta}{\alpha}:=\{f (\alpha)  \,:\, f:\alpha \rightarrow \beta \text{ is strictly increasing} \}.$$
The previous is commonly denoted by $[\beta]^{\alpha}$.   Then for any $n<\omega$  and any finite coloring of
$\binom{\omega}{n}$ there is an isomorphic copy $M$ of $\omega$ (i.e. the image of a strictly increasing
$f:\omega\rightarrow\omega$) and some
   $I\subseteq n(:=\{0,1,\dots,n-1\})$ such that any two  $n$-element subsets have the same color  if and only if they agree
  on the corresponding relative positions given by $I$.

This   was extended by P. Pudl\'ak and V. R\"odl in \cite{Pu-Ro}  for colorings of  a given \emph{uniform}
family $\mathcal{G}$ of finite subsets of $\omega$ (see  Section \ref{moredefn}) by showing that given any
coloring of  $\mathcal{G}$, there exists $A$ an infinite subset of $\omega$, a
uniform family $\mathcal{T}$ and a mapping $f:\mathcal{G}\to \mathcal{T}$ such that $f(X)\subseteq X$ for all
$X\in \mathcal{G}$ and such that any two $X,Y\in \mathcal{G}\upharpoonright A$ have the same color   if and
only if $f(X)=f(Y)$.

There is a natural extension of the Erd\H{o}s-Rado result, a kind of two-dimensional result  for certain
trees.    Let us define a \textit{\(b\)-branching tree} as a rooted tree \((T,<)  \) of height at most
\(\omega\) with the properties that   for every non-terminal node \(t\) the set of immediate successors
\(T_{t}\)\ has cardinality \(b\) and it is equipped with a fixed linear ordering \(<_{t}\), and such that the
terminal nodes (if any) have all the same height.  Examples of them are, given \(\tau\le\omega\), the tree
\((b^{<\tau},<)\) of functions \(f:i\rightarrow b\), \(i<\tau\),  endowed with the extension of functions
ordering \(<\), and ordering the set of immediate successors of a given \(f\) naturally. 
It is easy to see
that for any \(b\)-branching trees \(T\) and \(U\) of the same height there is a  unique
\textit{lexicographical-isomorphism} \(i_{T,U}:T\rightarrow U\), i.e. a tree-isomorphism  preserving the
corresponding orderings on sets of immediate successors (see Section $3$). In fact \((b^{<\tau},<)\) are the only examples, up to isomorphism, of $b$-branching trees with all terminal nodes of the same height.
Given two \(b\)-branching trees \(T\) and \(U\),
\textit{a strong embedding} is a lexicographical-isomorphic embedding   $i:T\rightarrow U$  which is  level
and meet preserving, that is, if $s,t \in T$ have the same height then also $i(s)$ and $i(t)$ and the meet
\(i(s)\wedge i(t)  \) of \(i(s)\) and \(i(t)\) is \(i(s\wedge t)\). For a definition of $s\wedge t$ see Section $3$.
 In this case, we say that \(i(T)\) is a
strong subtree of \(U\) isomorphic to \(T\).  Let     $\binom{U}{T}$ denote the family of strong-subtrees  of
$U$ isomorphic to \(T\). Then it is proved by K. Milliken \cite{Mi1} (see Section $3$) that for every finite coloring of
$\binom{b^{<\omega}}{b^{<n}}$  there is $T\in \binom{b^{<\omega}}{b^{<\omega}}$ such that the coloring on
$\binom{T}{b^{<n}}$ is constant. Notice that when $b=1$, then the result is exactly the Ramsey theorem for
$[\omega]^n$.
 In an unpublished paper, Milliken
\cite{Mi2} extended the Erd\H{o}s-Rado canonization theorem  by proving that given $n$ and an arbitrary
coloring $c:\binom{b^{<\omega}}{b^{<n}}\to \omega$, there is $T\in \binom{b^{<\omega}}{b^{<\omega}}$ and there
are a set of levels $I\subseteq n$ and a set of nodes $J\subseteq b^{< n}$ such that for every $T_0,T_1\in \binom{T}{b^{<n}}$ one
has that $c(T_0)=c(T_1)$ if and only if the $i$-th level of $T_0$ and of $T_1$ sit in the same level of $T$
(equivalently of $b^{<\omega}$) for every $i\in I$, and if for every \(t\in J\) the \(t\)th position of \(T_0\) and of \(T_1\) are the same, i.e.,
$i_{b^{<n},T_0}(t)=i_{b^{<n},T_1}(t)$.

In this paper we define properly the notion of uniform family of finite strong subtrees of a given infinite \(b\)-branching tree $U$,
and then we extend Milliken's result by proving the Pudl\'ak-R\"odl  canonization analogue for such uniform
families. More precisely,
 our main result Theorem \ref{maintheorem} in Section $6$ is the following.
  \begin{mtheo} Given any coloring of a uniform family of finite strong subtrees of $U$, there exists a
strong subtree $T$ of $U$ and a family of node-level sets, so that any two finite strong subtrees of the uniform
family have the same color if and only if they agree on one of these node-level sets.
\end{mtheo}

The proof is by induction on the complexity of the given uniform family, and Lemma $8$ is the natural
version of the corresponding result used by Pudl\'{a}k and R\"{o}dl to derive their theorem. Roughly tells  that
given any two uniform families $\mathcal{S}$ and $\mathcal{T}$ on $U$ and two mappings $f: \mathcal{S} \to R$
and $g:\mathcal{T}\to R$, there is a strong subtree $T$ of $U$ such that either $\mathcal{S}\upharpoonright
T=\mathcal{T}\upharpoonright T$ and $f \upharpoonright (\mathcal S\upharpoonright T)=g \upharpoonright
(\mathcal T\upharpoonright T)$, or else   $f(\mathcal{S}\upharpoonright T)\cap g(\mathcal{T}\upharpoonright
T)=\emptyset$.

The paper is organized as follows:

In the beginning, Section $2$, we present the results of Erd\H{o}s-Rado and Pudl\'ak-R\"odl to provide the reader with some intuition as they form particular cases of our Main Theorem. Then, in Section $3$, we introduce the notion of a uniform family of finite strong subtrees, given an infinite $b$-branching tree $U$. We give all the elementary properties and then we state the results of Milliken. Next, in Section $4$, we show that $\mathcal{S}_{\infty}((U_i)_{i\in \omega})$, the set of all infinite strong subtrees of a $d$-sequence of $b$-branching trees, forms a topological Ramsey space, a fact that is used in the proof of our Main Theorem that is stated and proved in the last section.

\section{Canonical Ramsey theorems of  Erd\H{o}s-Rado and Pudl\'ak-R\"odl}
Let $\mathcal{G}$ be a family of finite subsets of $\omega$.  We say that $\mathcal{G}$ is {\it Ramsey} when
for every partition $\mathcal{G}=\mathcal{G}_1\cup \mathcal{G}_2$, there is an infinite subset $X\subseteq
\omega$ and some $i\in \{1,2\}$ such that the restriction $\mathcal{G}_i\upharpoonright X:=\{ s\in
\mathcal{G}_i\,:\, s\subset X\}$ of $\mathcal G_i$ to $X$ is empty.   As one can expect,  not just any family of
finite subsets is Ramsey. A trivial example of a non Ramsey family is $[\omega]^{\le n}:=\{s\subset \omega\,:\,|s|\le n\}$ for $n>
1$.    Remarkably, C.
 Nash-Williams intrinsically characterizes the Ramsey property as follows.
\begin{theorem}[Nash-Williams, \cite{Na-Wi}] Let    $\mathcal{G}$ be a family of finite subsets of $\omega$.
\begin{enumerate}
\item[(a)] Suppose  that  $\mathcal{G}\upharpoonright X$ is {\it thin}; that is,
 there are no $s,t\in \mathcal{G}\upharpoonright X$ such that $s$ is a proper initial segment of $t$.
 Then $\mathcal G$ is Ramsey.
 \item[(b)] Suppose that $\mathcal G$ is Ramsey. Then there is some $X$ such that $\mathcal G \upharpoonright
 X$ is thin.
\end{enumerate}
\end{theorem}
Given a family $\mathcal G$ on $\omega$ and  $n\in \omega$, let
$$\mathcal{G}(n)=\{ A\subset \omega |\, \{n\}\cup A\in \mathcal{G} \text{ and } n<\min A\}.$$
We pass now to recall the notion of $\alpha$-uniform families on some infinite set $X$.

\begin{definition}[Pudl\'ak-R\"odl] Let $\mathcal G $ be a family of finite sets of an infinite subset
$X$ of $\omega$, and let $\alpha$ be a countable ordinal number. The family $\mathcal G$ is called
$\alpha$-uniform when
\begin{enumerate}
\item[(a)] $\mathcal{G}=\{\emptyset\}$ if $\alpha=0$;
\item[(b)] $\emptyset \notin \mathcal{G}$, $\mathcal{G}(n)$ is $\beta$-uniform on $X\setminus (n+1)$ for every $n\in X$, if $\alpha=\beta+1$;
\item[(c)] $\emptyset \notin \mathcal{G}$, there is an increasing sequence $(\alpha_n)_n$ with limit $\alpha$ such that each $\mathcal {G}(n)$ is
$\alpha_n$-uniform on $X\setminus (n+1)$, if $\alpha$ is a limit ordinal.
\end{enumerate}

 \end{definition}
It is easy to see that the only   $n$-uniform families on $X$ are $[X]^n:=\{s\subset \omega\,:\, |s|=n\}$ for
$n\in \omega$. For $\alpha\ge \omega $ this is not the case (consider for example the two $\omega$-uniform
families on $\omega$ $\{s\subset \omega\,:\, |s|=\min s +1\}$ and $\{s\subset \omega\,:\, |s|=\min s +2\}$).

Notice that if $\mathcal{G}$ is an $\alpha$-uniform family on $X$, then for any infinite subset $Y$ of $X$,
the restriction $\mathcal{G}\upharpoonright Y$ is also an $\alpha$-uniform family on $Y$. Also if
$\mathcal{G}$ is a uniform family, then it is  Nash-Williams as well. The relevance of uniform families is
given by the following.

\begin{lemma}\cite{Pu-Ro} For every  family $\mathcal{G}$ on $X$ there exists $Y\subseteq X$ such that either
   $\mathcal{G}\upharpoonright Y=\emptyset$ or $\mathcal{G}\upharpoonright Y$ contains a uniform family on $Y$.
 \end{lemma}

  To state the canonization result by Pudl\'ak and R\"odl we need the following definition which will be later extended in Definition \ref{canonical}  to the context of trees.

\begin{definition}
Let  $\mathcal{G}$ be a uniform family on some set \(X\). A coloring \(c\) of \(\mathcal{G}\) is called a \emph{canonical coloring} of  $\mathcal{ G }$ if there exists a uniform family $\mathcal{T}$ on \(X\) and a mapping $f:\mathcal{G}\to \mathcal{T}$ such that
\begin{enumerate}
\item[(a)] $f$ is \emph{inner}, i.e.  $f(s)\subseteq s$ for every $s\in \mathcal{G}$.
\item[(b)] For every $s,t\in \mathcal{G}$, $c(s)=c(t)$ if and only if $f(s)=f(t)$.
\end{enumerate}
\end{definition}
Notice that  the condition  (b) above  is equivalent to say that there exists  a one-to-one coloring $\phi$ of $\mathcal{T}$ with the same list of colors as that for the coloring of $\mathcal{G}$, such that $c(s)=\phi(f(s))$ for every $s\in \mathcal{G}$.

Roughly speaking $c$ is a canonical coloring of $\mathcal{G}$ if the color of each $s\in \mathcal{G}$ is determined by some subset $t$ of $s$ in a minimal way.
\begin{theorem}[Pudl\'ak-R\"odl,\cite{Pu-Ro}] \label{pud-rod} For every
coloring $c$ of a  uniform family $\mathcal{G}$ on   $X$, there exists $Y\subseteq X$  such that $c\upharpoonright (\mathcal{G}\upharpoonright Y)$ is a canonical coloring of $\mathcal{G}\upharpoonright Y$.
\end{theorem}

Given $A=(a_0,\dots,a_{n-1)},B=(b_0,\dots,b_{n-1}) \in [\omega]^n$ and $I\subseteq n$ we write $A:I=B:I$ to denote that $\{ a_i:i\in I\}=\{b_i:i\in I\}$.
In particular, for uniform families of finite rank the Erd\H{o}s-Rado Theorem follows from the Pudl\'ak-R\"odl Theorem.
\begin{theorem}[Erd\H{o}s-Rado,\cite{Er-Ra}] Given  $n\in \omega$ and a  mapping $c:[\omega]^n\to R$, there
exist an infinite subset $X\subseteq \omega$ and a finite set
 $I\subseteq n$ such that for any $A,B\in [X]^n$ one has $c(A)=c(B)$ if and only if  $A:I=B:I$.
\end{theorem}
The proof goes as follows. Use the Pudl\'ak-R\"odl Theorem to find some subset \(X\), some $k\le n$ and some inner
\(\phi:[X]^n\to [X]^{k}\) such that $c(s)=c(t)$ iff $\phi(s)=\phi(t)$. Now consider the finite coloring
\(d:[X]^n\to \mathcal{P}(n)\) defined by \(d(s):=I\subseteq n$ such that $s:I=\phi(s)$. By the Ramsey Theorem, there is a
subset \(Y\) of \(X\) and $I_0\subseteq n$ such that \(d\) is constant on \([Y]^{n}\) with value $I_0$. This just
means that $A$ and $B$ in $[Y]^n$ have the same $c$-color  if and only if $A$ and $B$ agree on the relative
positions given by $I_0$, denoted by
$$A:I_0=B:I_0.$$

The Pudl\'ak-R\"odl Theorem was proved by transfinite induction on the rank of the uniform family,
and it crucially uses the following lemma, that we will use later in our paper.
\begin{lemma}\cite{Pu-Ro}
Let $\mathcal{G}_1$ and $\mathcal{G}_2$ be two uniform families on $Y\subseteq \omega$, $\phi_1,\phi_2$ one- to-one
mappings defined on $\mathcal{G}_1$ and $\mathcal{G}_2$ respectively. Then there exists an infinite subset
$X\subseteq Y$ such that one of the following two statements holds:
\begin{enumerate}
\item[(1)] $\mathcal{G}_1\upharpoonright X=\mathcal{G}_2\upharpoonright X$ and $\phi_1(A)=\phi_2(A)$ for every $A\in \mathcal{G}_1\upharpoonright X$.
\item[(2)] $\phi_1(\mathcal{G}_1\upharpoonright X)\cap \phi_2(\mathcal{G}_2\upharpoonright X)=\emptyset$.
\end{enumerate}

\end{lemma}

\section{Uniform families of finite strong subtrees} \label{moredefn}
All the trees $U$ that we consider are rooted and have  height at most $\omega$.   For a given node $s\in U$ let
$|s|$ be its height in $U$, and similarly we write $|X|$ to denote the height of a subtree $X$ of $U$. Given
$n < |U|$, let $U(n)$ be the $n$th level of $U$, that is, the set of all nodes of $U$ of
height $n+1$.  Given $X\subseteq U$ let
$$L_X:=\{|s|-1\,:\, s\in X\} \subseteq L_U=\omega.$$
By $L_X<L_Y$ we mean that $\max L_X< \min L_Y$.
It is clear that in our context we can identify each node $s$ with the sequence of its predecessors. Given
$s,t\in U$ we write $s\wedge t$ to denote the \emph{meet} of $s$ and $t$, that is
$$s\wedge t:=\max_<\{u\,:\, u\le s,t\}.$$

To simplify the terminology we introduce the following concept.
\begin{definition}
Let $b>0$ be an integer. We call a tree $(U,<)$ a \emph{$b$-branching tree} when
\begin{enumerate}
\item[(a)] $U$  is rooted, and it has height at most $\omega$.
\item[(b)] All terminal nodes (if any) have the same height.
\item[(c)] For every non-terminal node $t\in U$ the set $U_{t}$ of immediate successors of $t$ has
cardinality $b$, and it is equipped with a total ordering $<_t$.
\end{enumerate}
\end{definition}
Notice that $b$-branching trees are naturally lexicographically well ordered by $s<_\mathrm{lex} t$ if and only if one of the following two possibilities holds.
\begin{enumerate}
\item{} The unique node $u_s$ in $U_{s\wedge t}$ below $s$ is $<_{s\wedge t}$ than the unique node $u_t$ in $U_{s\wedge
t}$ below $t$, where $<_{s\wedge t}$ is the prescribed linear ordering on $U_{s\wedge t}$. 
\item{}
 The two nodes satisfy $|s|<|t|$.
\end{enumerate}
The typical
$b$-branching tree is for $\tau\le \omega$ the set  $b^{<\tau}$ of mappings $f:n\to b$, $n<\tau$,
endowed with the ordering of extension of functions.
\begin{definition}
 Given two $b$-branching trees $U$ and $T$, an isomorphic embedding $\iota:U\rightarrow T$ is called
a \emph{strong embedding} when
\begin{enumerate}
\item[(1)] $\iota$ is $<_\mathrm{lex}$-preserving, i.e. if $s<_\mathrm{lex}t$ in $U$, then $\iota(s)<_{\mathrm{lex}}
\iota(t)$ in $T$;
\item[(2)] $\iota$ is meet-preserving, i.e. $\iota(s\wedge t)=\iota(s)\wedge \iota(t)$; and
\item[(3)] $\iota$ is level-preserving, i.e. if $|s|=|t|$ then $|\iota(s)|=|\iota(t)|$.
\end{enumerate}
$\iota$ is a strong isomorphism if it is a strong and onto embedding. In that case we call $U$, $T$ isomorphic and we denote $\iota_{U,T}:U\to T$ the strong isomorphism.
\end{definition}
The following is easy to prove.
\begin{proposition}
For every $b$-branching tree $U$ there is a unique $\tau\le \omega$ and a unique strong isomorphism  $\iota_{b^\tau,U}:b^{\tau}\rightarrow U$. Moreover such $\tau$ is the height of $U$.
\end{proposition}

\begin{definition}
Let  $U$ be  a $b$-branching tree and let $T\subseteq U$ be a $b$-branching subtree of $U$. We say that $T$
is a \emph{strong subtree} of $U$ when the inclusion mapping is a strong embedding.

Given $n \in \omega$, let $\mathcal S_{n}(U)$ be the family of all  strong subtrees of $U$ of height
$n$. By $\mathcal{S}_{\infty}(U)$ we denote the family of all strong subtrees of $U$ of infinite height.

Similarly for a $d$-sequence of $b$-branching trees $(U_i)_{i\in d}$ we call $(X_i)_{i\in d}$ a strong subtree of $(U_i)_{i\in d}$ if $X_i\in \mathcal{S}_{\tau}(U_i)$ and $L_{X_i}=L_{X_j}$ for all $i,j\in d$ and some $\tau\leq \omega$.
\end{definition}

Observe that nodes of $U$ are 1-strong subtrees of $U$

From now on, we fix an  infinite $b$-branching tree $U$. We are going to use letters $X,Y,Z,...$ and
$F,T,V,...$  to denote finite and infinite strong subtrees of $U$, respectively. Given strong subtrees $X,Y$
of $U$ by $X\sqsubseteq Y$ we mean that $X$ is an initial segment of $Y$, i.e. $X\subseteq Y$ and $Y(n)=X(n)$
for every $n<|X|$. Identical in the case of $Y=U$.
Similarly in the case of a $d$-sequence of $b$-branching trees $(U_i)_{i\in d}$ we call $(X_i)_{i\in d}$ and initial segment of $(Y_i)_{i\in d}$ if and only if $X_i\sqsubseteq Y_i$ for all $i\in d$. We denote the fact that $(X_i)_{i\in d}$ is an initial segment of $(Y_i)_{i\in d}$ by $(X_i)_{i\in d}\sqsubseteq (Y_i)_{i\in d}$.
We pass now to introduce operations for producing strong subtrees of $U$.

\begin{definition}
Given $t\in U$, let
$$U[t]=\{\, s\in U: t\leq s\,\}.$$
For $X\in \mathcal{S}_n(U)$ let
$$U[X]=\{ s\in U: \exists t\in X, t \leq s \,\}.$$
\end{definition}
So, $U[X]$ is the largest, under inclusion,  strong subtree of $U$ that has $X$ as initial segment.
Similarly for a given $t=(t_0, \dots, t_{n-1})\in \prod_{i\in d}U_i(n)$, let $$(U_i)_{i\in d}[t]=\{ U_i[t_i]\text{ for all } i\in d \}.$$

\begin{definition}
Let $Y$ be a finite strong subtree of $U$ of height $k$, and let $(T_i)_{i\in b^k}$ be a sequence of strong
subtrees of $U$ such that
\begin{enumerate}
\item[(a)] $L_{T_i}=L_{T_j}$ for every $i,j\in b^k$;
\item[(b)] The root of $T_i$ is different from the root of $T_j$ for every $i\neq j \in b^k$; and
\item[(c)] $\{T_i\}_{i\in [j\cdot b, (j+1) \cdot b^)}\subseteq U[t_j]$ for every $j<b^{k-1}$, where $\{t_j\}_{j\in
b^{k-1}}$ is the lexicographically ordered set of terminal nodes of $Y$.
\end{enumerate}

Set
$$Y^ \frown (T_i)_{i\in b^k}:=Y\cup \bigcup_{i\in b^k}T_i.$$
Given  a strong subtree $W$ of $U$ and given an initial part $Y$ of $W$ let $W(Y)$  be the unique sequence
$( Z_i)_{i\in d}$  of strong subtrees of $W$ such that  $Y^\frown (Z_i)_{i\in d}=W$.
\end{definition}
\begin{rem}
 Let $Y$ and $(T_i)_{i\in b^k}$ be as in Definition $7$. Let $\iota:b^{< k+\tau}\to U$ be the mapping defined by $\iota(s):=\iota_{b^k,Y}(s)$ for $s\in b^{<k}$ and $\iota(f):=\iota_{b^{\tau},T_{f(k)}}(\widehat{f})$ for $f\in b^{k+l}$, 
$l<\tau$, and where $\widehat{f}:l\to b$ is defined by $\widehat{f}(j):=f(k+j)$.  Then (a)-(c) above is equivalent to saying that $\iota$ is 
a strong embedding. 
\end{rem}

Whenever we write $Y^\frown (T_i)_{i\in b^k}$ we implicitly assume that (a)-(c) above hold.
 For a node $t\in W$
considered as a 1-strong subtree of $W$ we write $t^\frown (T_i)_{i\in b}$ and  $W[t]$ instead of $\{t\}^\frown
(T_i)_{i\in b}$ and  $W[\{t\}]$, respectively. For $t=(t_0, \dots, t_{n-1})\in \prod_{i\in d}U_i(n)$, $n\in \omega=L_{(U_i)_{i\in d}}$, and a $d\cdot b$-sequence of $b$-branching trees $(Y_j)_{j\in d\cdot b}$, we define $$t^\frown (Y_j)_{j\in d\cdot b}=\bigcup_{i\in d} t_i^\frown (Y_j)_{j\in [i\cdot b, (i+1)\cdot b)}.$$

Let $(Y_i)_{i\in d}\in \mathcal{S}_n((U_i)_{i\in d})$ and $(T_j)_{j\in d\cdot b^n}$. We define the $d$-sequence of trees $$((Y_i)_{i\in d})^\frown(T_j)_{j\in d\cdot b^n}=(Y_i^\frown (T_j)_{j\in [i\cdot b^n, (i+1)\cdot b^n)})_{i\in d}$$ an infinite strong subtree of $(U_i)_{i\in d}$.

Now for every node $t$ of $U$ we define a $b$-sequence of strong subtrees as follows:

$$U(t):=\{ (T_i)_{i\in b}: t^\frown (T_i)_{i\in b} =U[t]\}$$

Similarly for $t=(t_0, \dots, t_{d-1})\in \prod_{i\in d}U_i(n)$ we define a $d\cdot b$-sequence of strong subtrees as follows:

$$(U_i)_{i\in d}(t):=\{ (T_i)_{i\in d\cdot b}: t^\frown (T_i)_{i\in d\cdot b}  =(U_i)_{i\in d}[t]\}$$

\begin{definition}
Let $\mathcal{G}$ be a family of finite strong subtrees of $U$. Let $Y$ be a finite strong subtree of $U$ of
height $k$. We define
\begin{equation}
\label{iijegfef} \mathcal{G}(Y):=\{\, (Z_i)_{i\in b^{k}}: Y^\frown (Z_i)_{i\in
b^{k}}\in \mathcal{G}\, \}.
\end{equation}
For a node $t$ we write $\mathcal G(t)$ instead of $\mathcal G(\{t\})$. 
Given $t\in U$,
let
$$t^\frown \mathcal{G}(t):=\{ \, t^\frown (X_i)_{i\in d}: (X_i)_{i\in d}\in \mathcal{G}(t)\,\}\subset \mathcal{G}.$$ and given $i\in b$,
$$\pi_i (\mathcal{G}(t)):=\{X\in \mathcal S_{<\omega}(U)\, :\, \text{there is $(X_j)_{j\in b}\in \mathcal{G}(t)$ and $X_i=X$}\}.$$

Similarly for $t=(t_0, \dots, t_{n-1})\in \prod_{i\in d}U_i(n)$ and $\mathcal{G}$ a family of finite strong subtrees of $(U_i)_{i\in d}$, we define $$\mathcal{G}(t):=\{ \,  (X_j)_{j\in d\cdot b}: t^\frown (X_j)_{j\in d\cdot b}\in \mathcal{G}\,\}$$ and $$\mathcal{G}(t_i):=\{ \,  (X_j)_{j\in [i\cdot b,(i+1)\cdot b)}:  \exists (X'_j)_{j\in d\cdot b}\in \mathcal{G}(t), X'_j=X_j \text{ for all }j\in [i\cdot b,(i+1)\cdot b)\,\}.$$

\end{definition}

Finally, we are ready to define uniform families of finite strong subtrees of $U$ and of $(U_i)_{i\in d}$.  
\begin{definition}\label{uniform}
Let $\alpha$ be a countable ordinal number. We say that a family $\mathcal G$ of finite strong subtrees of $U$ is   \emph{$\alpha$-uniform} if the following hold.
 
\begin{enumerate}
 \item[(1)] If $\alpha=0$, then $\mathcal{G}=\{\emptyset\}$.
\item[(2)] If $\alpha=\beta+1$, then $\emptyset \notin \mathcal{G}$ and $\pi_i(\mathcal{G}(t))$ is $\beta$ uniform on $U[t^\frown i]$ for every $t\in U$ and $i\in b$.
\item[(3)] If $\alpha$ is a limit ordinal, then $\emptyset \notin \mathcal{G}$, and for all $t\in U$ and $i\in b$, there is some $\alpha_t< \alpha$ such that 
  $\pi_i(\mathcal{G}(t))$ is $\alpha_{t}$ uniform on $U[t^\frown i]$ and
\begin{enumerate}
 \item[(3.1)] $\{\, t\in U: \alpha_{t}=\beta\, \}$ is finite for every $\beta<\alpha$, and   
 \item[(3.2)] $\sup_{t\in C}\{\alpha_{t}\}=\alpha$ for every infinite chain $C$ of $U$.  
 
\end{enumerate}

 \end{enumerate}
 Similarly we define \emph{$\alpha$-uniform families of $d$-tuples $(X_i)_{i\in d}$ of finite strong subtrees of $(U_i)_{i\in d}$} as follows:
\begin{enumerate}
 \item[(1)] If $\alpha=0$, then $\mathcal{G}=\{\emptyset\}$;
\item[(2)]  If $\alpha=\beta+1$, then $\emptyset \notin \mathcal{G}$ and for every $t=(t_0, \dots, t_{d-1})\in \prod_{i\in d}U_i(n)$ one has
 that:\\ $(\pi_{j_i} \mathcal{G}(t_i))_{i\in d}$ on $ (U_i[t_i^\frown j_i])_{i\in d}$ is $\beta$-uniform, where for every 
$i\in d$, $j_i\in b$.
\item[(3)] If $\alpha$ is a limit ordinal, then $\emptyset \notin \mathcal{G}$ and for every $t=(t_0, \dots, t_{d-1})\in \prod_{i\in d}U_i(n)$, $n\in \omega$, one has
 that:\\ $(\pi_{j_i} \mathcal{G}(t_i))_{i\in d}$ on $ (U_i[t_i^\frown j_i])_{i\in d}$ is $\alpha_{t}$-uniform, where for every 
$i\in d$, $j_i\in b$ and 
\begin{enumerate}
\item[(3.1)] $\{\, t=(t_0,\dots, t_{d-1})\in \prod_{i\in d}U_i(n) :\alpha_t=\beta\,\}$ is finite for every $\beta<\alpha$,
\item[(3.2)] for any infinite chain $C$ of $\bigcup_{n\in \omega}\prod_{i\in d} U_i(n)$,  the tree that results by taking the level product of $(U_i)_{i\in d}$, we have that $(\alpha_t)_{t\in C}\to \alpha$.

\end{enumerate}

\end{enumerate}

   \end{definition}

The first thing that we remark is that by an easy inductive argument, if $\mathcal{G}$ is an $\alpha$-uniform family on $U$ and $T\in \mathcal{S}_{\infty}(U)$, then $\mathcal{G}\upharpoonright T=\{ X \in \mathcal{G}: X\in \mathcal{S}_n( T), n\in \omega \}$ is also $\alpha$-uniform on $T$.
   For $n\in \omega$ there is exactly one $n$-uniform family on $U$, the family of all strong subtrees of height $n$, namely $\mathcal{S}_n(U)$. It is easy to show that for each $\alpha\geq \omega$ there are infinitely many different $\alpha$-uniform families. A typical example of an $\omega$-uniform family on $U$ is the family $\mathcal{F}$ defined by $X\in \mathcal{F}$ if and only if the height of $X$ is equa tol the height of its root $r_X$. \\

           \subsection{ Canonical Ramsey Theorem of Milliken.}

    Recall the following pigeonhole principle for $\mathcal{S}_n(U)$.

    \begin{theorem}[Milliken,\cite{Mi1}]  Let $n,l$ be positive integers. For any finite coloring $c:\mathcal{S}_n(U)\to l$ of the $n$-uniform family of finite strong subtrees of $U$, there exists an infinite strong subtree $T$ of $U$ such that $c$ restricted on $\mathcal{S}_n(T)$ is constant.
  \end{theorem}

\begin{definition} Let $X$ and $Y$ be strong subtrees of $U$ of height $n$. Let $N\subseteq b^{n}$ be a \emph{node set}. We say that 
$X$ and $Y$ \emph{agree on $N$} when $\iota_{b^n,X}(s)=\iota_{b^n,Y}(s)$ for every $s\in N$.

Let $L\subseteq n$ be a  \emph{set of levels}.
We say that $X$ and $Y$ \emph{agree on $L$} if for every
 $l\in L$ the $l$th level of $X$ and the $l$th level of $Y$ both lie on the same level of $U$.  

For $N\subseteq b^{<n}$ and $L\subseteq n$ We write 
$$X:(N,L)=Y:(N,L)  $$
to denote that $X$ and $Y$ agree on the node-level set $(N,L)$.  
 \end{definition}
 
Extending the Erd\H{o}s-Rado Theorem, Milliken obtained the following:
   \begin{theorem}[Milliken,\cite{Mi2}]\label{milliken-canon} For any coloring $c$ of the $n$-uniform family of finite strong subtrees of $U$, there exists an infinite strong subtree $T$ of $U$ and a node-level set $(N,L)$ so that for any $X,Y\in \mathcal{S}_n(T)$ one has $c(X)=c(Y)$ if and only if $$X:(N,L)=Y:(N,L)$$ For the above pair it holds that $L_N< L$, that is, the levels of $b^n$ on which the nodes of $N$ lie are strictly less than the levels appearing in $L$.
   \end{theorem}
Observe that in the case of the uniform family of rank one, namely $\mathcal{S}_1(U)$, the above theorem gives us an infinite strong subtree $T$ of $U$ such that the coloring $c$ is constant $(N=L=\emptyset)$, one-to-one $(N=b^1$, $L=\emptyset)$, or constant on the levels $(N=\emptyset$, $L=\{0\})$, i.e. $c(t)=c(s)$ if and only if $|t|=|s|$.

 We assume from now on that for any uniform family of infinite rank  $\mathcal{G}$, that we consider, the rank of each uniform family $\mathcal{G}(t)$ on $U(t)$, for every node $t$, follows the lexicographic ordering $(U, <_\mathrm{lex})$ introduced above, i.e. for $s<_\mathrm{lex} t$ we have that the rank of $\mathcal{G}(s)$ on $U(s)$ is less than or equal the  rank of $\mathcal{G}(t)$ on $U(t)$. This is obvious if $\alpha$ is a successor ordinal. If $\alpha$ is a limit ordinal, then by Definition $9(3.1)$ we have that the set $\{\, t\in U: \alpha_{t}=\beta\, \}$ is finite for every $\beta<\alpha$. Consider the coloring $c: \mathcal{S}_1(U)\to \alpha$ defined by $c(t)=\beta$ if $\mathcal{G}(t)$ is of rank $\beta<\alpha$. By Theorem $5$ there exists $T\in \mathcal{S}_{\infty}(U)$ such that $c\upharpoonright \mathcal{S}_1(T)$ is either one-to one, or constant on the levels. In both cases the rank of each uniform family $\mathcal{G}(t)$ on $T(t)$, for every node $t$, follows the lexicographic ordering $(T, <_\mathrm{lex})$, modulo passing to an infinite strong subtree.

   Notice that Theorem $5$ is an analog, 
in some sense, of the Pudl\'ak-R\"odl theorem and extends the finite version of Milliken's theorem. Our main theorem of this paper is going to extend 
Theorem \ref{milliken-canon} to an arbitrary uniform family,
 completing the analog between Erd\H{o}s-Rado and Pudl\'ak-R\"odl. Before stating the main theorem we still need some new concepts and results.

   \section{$\mathcal{S}_{\infty}((U_i)_{i\in \omega})$ as topological Ramsey space}
  
   We introduce the notion of Nash-Williams on families of finite strong subtrees. We remind the reader the notion of initial segment. Given strong subtrees $X,Y$
of $U$ by $X\sqsubseteq Y$ we mean that $X$ is an initial segment of $Y$, i.e. $X\subseteq Y$ and $Y(n)=X(n)$
for every $n<|X|$. Identical in the case of $Y=U$.
Similarly in the case of a $d$-sequence of $b$-branching trees $(U_i)_{i\in d}$ we call $(X_i)_{i\in d}$ and initial segment of $(Y_i)_{i\in d}$ if and only if $X_i\sqsubseteq Y_i$ for all $i\in d$. We denote the fact that $(X_i)_{i\in d}$ is an initial segment of $(Y_i)_{i\in d}$ by $(X_i)_{i\in d}\sqsubseteq (Y_i)_{i\in d}$.

   \begin{definition} A family $\mathcal{F}$ of finite strong subtrees of $U$ is {\it Nash--Williams} if given any two $X,Y\in \mathcal{F}$, $X$ is not an initial segment of $Y$.
   \end{definition}
   The first thing we notice is the following lemma:
       
      \begin{lemma} If $\mathcal{G}$ is uniform of $(U_i)_{i\in d} $, then $\mathcal{G}$ is Nash-Williams. \end{lemma}
      
      \begin{proof}
      By induction on $\alpha$ such that $\mathcal{G}$ is $\alpha$-uniform. 
      If $\alpha=0$, then the assertion is trivial.
       Let $\alpha> 0$ and assume the assertion holds for every $\beta<\alpha$. Assume that there are $(X_i)_{i\in d}$, $(Y_i)_{i\in d} \in \mathcal{G}$ and $(X_i)_{i\in d}\sqsubseteq (Y_i)_{i\in d}$. Let  $t=(t_i)_{i\in d}$ be the common root of $(X_i)_{i\in d}$ and $(Y_i)_{i\in d}$. By definition of uniform family for all $i\in b$ and $j_i\in b$, $(\pi_{j_i} \mathcal{G}(t_i^\frown j_i))_{i\in d}$ is a $\beta$-uniform family on  $ (U_i[t_i^\frown j_i])_{i\in d}$, for $\beta<\alpha$. From our assumption it follows that $(X_i[t_i^\frown j_i])_{i\in d}$ is an initial segment of $(Y_i[t_i^\frown j_i])_{i\in d}$ contradicting the inductive hypothesis. Therefore $\mathcal{G}$ has the property that for any two $(X_i)_{i\in d},(Y_i)_{i\in d} \in \mathcal{G}$ is not the case that $(X_i)_{i\in d}$ is an initial segment of $(Y_i)_{i\in d}$.
       \end{proof}

  The following lemma has an easy proof by induction on $\alpha$\\

    \begin{lemma} If $\mathcal{G}$ is $\alpha$-uniform on $(U_i)_{i\in d}$ then $\mathcal{G}\upharpoonright (T_i)_{i\in d}$ is also $\alpha$-uniform on $(T_i)_{i\in d}$, for any $(T_i)_{i\in d}\in\mathcal{S}_{\infty}((U_i)_{i\in d})$ \end{lemma}
     
 Now we introduce the notion of Ramsey on families of finite strong subtrees.

      \begin{definition} A family of finite strong subtrees $\mathcal{G}$ on $(U_i)_{i\in d}$ is Ramsey if for every finite partition $\mathcal{G}= \mathcal{G}_0\cup \dots \cup \mathcal{G}_{l-1}$ there exists $(T_i)_{i\in d}\in \mathcal{S}_{\infty}((U_i)_{i\in d})$ such that at most one of the sets $\mathcal{G}_i\upharpoonright (T_i)_{i\in d}$ is non empty.\end{definition}
      
      \begin{lemma}
      Any $\alpha$-uniform family $\mathcal{G}$ on $(U_i)_{i\in d}$ is Ramsey.
      \end{lemma}
    Before proving this Lemma we show that the family $\mathcal{S}_{\infty}((U_i)_{i\in d})$ forms a topological Ramsey space in the sense of \cite{To}. The reader is assumed to be familiar with the Theory of topological Ramsey spaces as presented in \cite{To}.  In \cite{To} Chapter $6$, it is shown that $\mathcal{S}_{\infty}(U)$ forms a topological Ramsey space, here we extend that argument in the case of finite sequences of trees. For $(X_i)_{i\in d} \in \mathcal{S}_n((U_i)_{i\in d})$, $n\in \omega$ and $(T_i)_{i\in d}\in \mathcal{S}_{\infty}((U_i)_{i\in d})$ we define: $$(T_i)_{i\in d}\upharpoonright n=\Big(\bigcup_{m<n}(T_i(m))_{i\in d} \Big), \text{ and}$$
    
    $[(X_i)_{i\in d}, (T_i)_{i\in d}]=\{\, (T'_i)_{i\in d}\in \mathcal{S}_{\infty}((T_i)_{i\in d}): (T'_i)_{i\in d}\upharpoonright n=(X_i)_{i\in d}\, \}.$\\ With that definition $\mathcal{S}_{\infty}((U_i)_{i\in d})$ becomes a topological space where the above sets are its basic open sets.
     
     For $(T_i)_{i\in d}\in \mathcal{S}_{\infty}((U_i)_{i\in d})$ the sequence $r_n((T_i)_{i\in d})$ of finite approximations (restrictions) is defined as follows:\\
     \begin{center} $r_n((T_i)_{i\in d})= (T_i)_{i\in d}\upharpoonright n$ \end{center}
     
     Thus the set of all finite approximations to elements of $\mathcal{S}_{\infty}((U_i)_{i\in d})$ is the set
     \begin{center} $\mathcal{S}_{<\infty}((U_i)_{i\in d})=\bigcup_{n\in \omega} \mathcal{S}_n((U_i)_{i\in d})$ \end{center}
     
of strong subtrees of $(U_i)_{i\in d}$ of finite height. The inclusion order on $\mathcal{S}_{\infty}((U_i)_{i\in d})$ is finitized as follows:\\

     \begin{center} $(X_i)_{i\in d}\subseteq_{fin} ( Y_i)_{i\in d}$ iff $(X_i)_{i\in d}=(Y_i)_{i\in d}=\emptyset$ or $(X_i)_{i\in d}\subseteq (Y_i)_{i\in d}$ and $(X_i)_{i\in d}(\max)\subseteq (Y_i)_{i\in d}(\max)$ \end{center} where $(X_i)_{i\in d}(\max)$ and $(Y_i)_{i\in d}(\max)$ denote the maximal levels of the strong subtrees $(X_i)_{i\in d},(Y_i)_{i\in d}$ respectively. Finitized in this way the space $$(\mathcal{S}_{\infty}((U_i)_{i\in d}),\subseteq, r)$$ is easily seen to satisfy the following list of axioms:\\

        $\bf{A.1}$
  \begin{enumerate}
         \item{}  $r_0((X_i)_{i\in d})=r_0((Y_i)_{i\in d})$ for all $(X_i)_{i\in d},(Y_i)_{i\in d}\in \mathcal{S}_{<\infty}((U_i)_{i\in d})$;\\
          \item{}$(X_i)_{i\in d}\neq (Y_i)_{i\in d}$ implies that $r_n((X_i)_{i\in d})\neq r_n((Y_i)_{i\in d})$ for some $n$;\\
         \item{} $r_n((X_i)_{i\in d})=r_m((Y_i)_{i\in d})$ implies $n=m$ and $r_k((X_i)_{i\in d})=r_k((Y_i)_{i\in d})$ for all $k\leq n$.
     \end{enumerate}
     
       $\bf{A.2}$
     \begin{enumerate}
     \item{} $\{\,(X_i)_{i\in d}\subseteq_{fin} (Y_i)_{i\in d}\, \}$ is finite for all $(Y_i)_{i\in d}$;\\
     \item{} $(T^0_i)_{i\in d}\subseteq (T^1_i)_{i\in d}$ iff $\forall n\, \exists m$ $r_n((T^0_i)_{i\in d})\subseteq_{fin} r_m((T^1_i)_{i\in d})$;\\
     \item{} $\forall (X_i)_{i\in d}, (Y_i)_{i\in d}$ $[(X_i)_{i\in d}\sqsubseteq (Y_i)_{i\in d} \wedge (Y_i)_{i\in d} \subseteq_{fin} (Z_i)_{i\in d}$ implies $\exists (W_i)_{i\in d}\sqsubseteq (Z_i)_{i\in d}\text{ such that } (X_i)_{i\in d}\subseteq_{fin} (W_i)_{i\in d}]$.
     \end{enumerate}
     
   $\bf{A.3}$
   \begin{enumerate}
     \item{} If $[(X_i)_{i\in d}, (T_i)_{i\in d}]\neq \emptyset$ then $[(X_i)_{i\in d}, (T'_i)_{i\in d}]\neq \emptyset$ for all $(T'_i)_{i\in d}\in [(X_i)_{i\in d}, (T_i)_{i\in d}]$;\\
     \item{} $(T^0_i)_{i\in d}\subseteq (T^1_i)_{i\in d}$ and $[(X_i)_{i\in d}, (T^0_i)_{i\in d}]\neq \emptyset$ imply that there exists $(T'_i)_{i\in d}\in [(X_i)_{i\in d}, (T^1_i)_{i\in d}]$ such that  $$ \emptyset \neq[(X_i)_{i\in d}, (T'_i)_{i\in d}]\subseteq   [(X_i)_{i\in d}, (T^0_i)_{i\in d}].$$
     \end{enumerate}
    
     The following requirement, that forms the pigeon hole principle in our case, requires some proof.\\

      $\bf{A.4}$\\
  Let $\mathcal{O}\subseteq \mathcal{S}_{l+1}((U_i)_{i\in d})$ and $[(X_i)_{i\in d}, (T_i)_{i\in d}]\neq \emptyset$, where the height of $(X_i)_{i\in d}$ is $l$ and we assume that $(T_i)_{i\in d}\upharpoonright l=(X_i)_{i\in d}$. There exists $(T'_i)_{i\in d}\in [(X_i)_{i\in d}, (T_i)_{i\in d}]$ such that $r_{l+1}[(X_i)_{i\in d}, (T'_i)_{i\in d}]\subseteq \mathcal{O}$ or $r_{l+1}[(X_i)_{i\in d}, (T'_i)_{i\in d}]\subseteq \mathcal{O}^{c}$. Where
  \begin{eqnarray*}
     r_{l+1}[(X_i)_{i\in d}, (T'_i)_{i\in d}]=\{ (Y_i)_{i\in d}\in \mathcal{S}_{l+1}((U_i)_{i\in d}) &:& (Y_i)_{i\in d}=(T''_i)_{i\in d} \upharpoonright l+1\text{ for } \\ (T''_i)_{i\in d}\in [(X_i)_{i\in d}, (T'_i)_{i\in d}] \}.
     \end{eqnarray*}

       \begin{proof} Let $ u_0,\dots, u_{p-1}$ be a one-to-one enumeration of the set of nodes of $\bigcup_{i\in d} U_i$ that are immediate successors of some node of the set $\{ \bigcup_{i\in d}X_i(l-1)\}$. For $j\in p$, let: $V_j=\{ t\in U_i: u_j\leq t\}$, where $i$ is such that $u_j\in U_i$.
    Note that every ${t}=(t_0,\dots, t_{p-1})\in \prod_{j\in p} V_j(k)$, for some $k\in \omega$, determines the strong subtree
    $$b({t})=(T_i)_{i\in d}\upharpoonright l\cup (t_0,\dots, t_{p-1})$$ of $(T_i)_{i\in d}$ of length $l+1$. Let $\mathcal{O}^{\star}=\{{t}:  b({t})\in \mathcal{O}\}$. \\
    By the strong subtree version of Halpern L\"auchli theorem (\cite{Ha-Lau}, \cite{To} Theorem 3.2), there is a sequence of strong subtrees  $(F_j)_{j\in p} \in \mathcal{S}_{\infty}( (U_i[u_j])_{j\in p})$, all with the same level sets, such that: $\bigcup_{n\in \omega} \prod_{j\in p} F_j(n)$ is a subset of either $\mathcal{O}^{\star}$ or its complement.
    Let: $(T'_i)_{i\in d}=((T_i)_{i\in d}\upharpoonright l)^\frown (F_j)_{j\in p}$.
    Then $(T'_i)_{i\in d}$ is a strong subtree of $(U_i)_{i\in d}$  that belongs to the basic open set $[(X_i)_{i\in d}, (T_i)_{i\in d}]$ such that $r_{l+1}[(X_i)_{i\in d}, (T'_i)_{i\in d}]$ is included either in $\mathcal{O}$ or its complement.  \end{proof}
    
    Therefore the space $(\mathcal{S}_{\infty}((U_i)_{i\in d}), \subseteq, r)$ forms a topological Ramsey space. We provide to the reader a brief explanation of  what it means $(\mathcal{S}_{\infty}((U_i)_{i\in d}), \subseteq, r)$ to be a topological Ramsey space. We say that a subset $\mathcal{X}$ of $\mathcal{S}_{\infty}((U_i)_{i\in d})$ is \emph{Ramsey} if for every $[(Y_i)_{i\in d}, (V_i)_{i\in d}]\neq \emptyset$ there is a $(F_i)_{i\in d}\in [(Y_i)_{i\in d}, (V_i)_{i\in d}]$ such that  either $[(Y_i)_{i\in d}, (F_i)_{i\in d}]\subset \mathcal{X}$ or $[(Y_i)_{i\in d}, (F_i)_{i\in d}]\subset \mathcal{X}^c$, and    $\mathcal{X}$  is \emph{Ramsey null} if for every $[(Y_i)_{i\in d}, (V_i)_{i\in d}]\neq \emptyset$, there is $(F_i)_{i\in d}$ such that $[(Y_i)_{i\in d}, (F_i)_{i\in d}]\cap \mathcal{X}=\emptyset$. Being a topological Ramsey space it means that Ramsey subsets of $\mathcal{S}_{\infty}((U_i)_{i\in d})$ are exactly those with the Baire property and that meager sets are Ramsey null. 
    
      Recall that a mapping $f:A \to B$ between two topological spaces is Suslin measurable, if the preimage $f^{-1}(O)$ of every open subset $O$ of $B$ belong to the minimal $\sigma-$field of subsets of $A$ containing closed sets and being closed under the Suslin operation, see \cite{Ke}. 
      
      As a consequence of the fact that $(\mathcal{S}_{\infty}((U_i)_{i\in d}), \subseteq, r)$ forms a topological Ramsey space is that its field of Baire measurable subsets coincides with that of Ramsey and is closed under the Suslin operation. Therefore for any finite coloring, where each color is Suslin measurable, the assertion of the following theorem is immediate.

    \begin{theorem}  For every finite Suslin measurable coloring of the set $\mathcal{S}_{\infty}((U_i)_{i\in d})$, there exists a strong subtree $(T_i)_{i\in d}\in \mathcal{S}_{\infty}((U_i)_{i\in d})$ such that $\mathcal{S}_{\infty}((T_i)_{i\in d})$ is monochromatic\end{theorem}

     The first consequence is the following:
     \begin{coroll} For every $\mathcal{F}\subseteq \mathcal{S}_{<\infty}((U_i)_{i\in d})$, there is a strong subtree $(T_i)_{i\in d}$ of $(U_i)_{i\in d}$ such that either
     \begin{enumerate}
     \item{} $\mathcal{S}_{<\infty}((T_i)_{i\in d})\cap \mathcal{F}=\emptyset$ or
     \item{} For every $(T'_i)_{i\in d}\in \mathcal{S}_{\infty}((T_i)_{i\in d})$ there is some $n$ such that $(T'_i)_{i\in d} \upharpoonright n\in \mathcal{F}$.
     \end{enumerate}
     \end{coroll}
     \begin{proof} Color elements of $ \mathcal{S}_{\infty}((U_i)_{i\in d})$ according to whether they have a restriction in $\mathcal{F}$ or not. This is a Borel coloring. Now apply Theorem $6$.\end{proof}
     
     We give now a proof for Lemma $5$.
        
     \begin{proof} Let $\mathcal{G}$ be an $\alpha$-uniform family on $(U_i)_{i\in d}$. By Lemma $3$, $\mathcal{G}$ is Nash-Williams. Let $G_0\cup \dots \cup G_{l-1}$ be a finite partition of $\mathcal{G}$. Apply the previous corollary successively to each of the colors.
     
     \end{proof}
     
     Therefore, any $\alpha$-uniform family $\mathcal{G}$ on $(U_i)_{i\in d}$ is Ramsey.

         \section{Strong subtree envelopes}
         
            At this point we would like to introduce a key notion of this paper, the strong subtree envelope of a given subset of $U$. This notion is discussed in \cite{To}.

 We recall that for $s,t\in U$, we have defined:

 $$s\wedge t = \max \{ \, u\in U: u\leq s\text{ and }   u\leq t\, \}.$$
 The $\wedge$-closure of $A\subseteq U$ is the set:
    $$A^\wedge=\{ \, s\wedge t: s,t \in A\, \}.$$
    We point out that in the definition of $A^\wedge$ $s$ can be equal to $t$. Note that $A\subseteq A^\wedge$ and that  $A^\wedge$ is a rooted tree. Finally,     for $A\subseteq U$, let
    $$||A||=|\{\, |s\wedge t|: s,t\in A\,\}|$$
    be the number of levels of $U$ which $A^\wedge$ intersects. 
    \begin{definition}
    The \emph{strong subtree envelope} of a node set $A\subseteq U$ is the following subset of $\mathcal{S}_{||A||}(U)$ defined by: \\
    \begin{center} $\mathcal{C}^U_A=\{\, X\in \mathcal{S}_{||A||}(U): A^\wedge\subseteq X\,\}.$\end{center}

    Notice that if $X,Y\in \mathcal{C}^U_A$, then $L_X=L_Y$ and also $ i_{b^{||A||},X}\circ i^{-1}_{ b^{||A||},Y}$ is the identity on $A$.

For a given  finite level set $L\subseteq L_U=\omega$, its strong subtree envelope is defined by:
         $$\mathcal{C}^U_{L}=\{ X\in \mathcal{S}_{|L|}(U):\, L_X=L\}.$$
If in addition $L$ is such that  such that  $L_A< L$, then we define
          \begin{center} $\mathcal{C}^U_{(A,L)}=\{\, X\in \mathcal{S}_{(||A||+|L|)}(U): A^\wedge\subset X$ and the last $|L|$ many levels of $X$ lie on the levels of $U$ indicated by $L \},$\end{center}
i.e., $\mathcal{C}^U_{(A,L)}$ is the set of all $X\in\mathcal{S}_{(||A||+|L|)}(U) $ such that $A^\wedge\subset X$ and such that for every  $i\in |L|  $  one has that $X(||A||+i)\subset U(l_i)$, where $\{l_0,\dots,l_{|L|-1}\}$ is the increasing enumeration of $L$.
      \end{definition}

        Similarly, given a finite sequence of trees $(U_i)_{i\in d}$ we define the strong subtree envelope of $(N_i,L_i)_{i\in d}$ in $(U_i)_{i\in d}$, where for all $i\in d$, $N_i\subset U_i$, $L_i\subset L_{U_i}$ and $L_{N_i}< L_i$, as follows:
    $$ \mathcal{C}^{(U_i)_{i\in d}}_{(N_i,L_i)_{i\in d}}=\{\, (X_i)_{i\in d}\in \mathcal{S}_n((U_i)_{i\in d})  \, : \,  \text{  $\forall i\in d$   $\exists Y_i\in  \mathcal{C}^{U_i}_{(N_i,L_i)}$ with $ Y_i\subseteq  X_i$}  \},$$  where $n=| \bigcup_{i\in d}(L_{N^\wedge_i}\cup L_i)|$.\\

    We make the observation that if $(X_i)_{i\in d}\in \mathcal{C}^{(U_i)_{i\in d}}_{(N_i,L_i)_{i\in d}}$ then $X_i$ is not necessarily a member of $\mathcal{C}^{U_i}_{(N_i,L_i)}$.

    We introduce now the notion of a translation of a strong subtree.

        \begin{definition}
   Let $X$ be a strong subtree of $U$ of finite height with root $r_X$, by a $\mathrm{translation}$ of $X$ we mean a strong subtree $Y$ of $U$, with root $r_Y\neq r_X$ such that the following two conditions hold:
   \begin{enumerate}
   \item{} $L_Y=L_X$;
   \item{} for every node $t\in X$ there is a corresponding node $s\in Y$ with $|s|=|t|$, and if $s,t$ are viewed as finite sequences of $\{0,\dots, b-1\}$, then $t\upharpoonright (|t|\setminus |r_X|)= s\upharpoonright (|s|\setminus |r_Y|)$.
     \end{enumerate}
     
        \end{definition}
        
    In other words we allow strong subtrees to be translated horizontally.

     For a subset $A$ of $U$ its translation is obtained as follows: Let $X\in \mathcal{C}^U_A$ and $Y$ be a translation of $X$. Set $ i_{ b^{||A||},Y}\circ i^{-1}_{b^{||A||},X}(A)$ a translation of $A$.

     Similarly we define translation in the context of a $d$-sequence of $b$-branching trees $(U_i)_{i\in d}$. For $(X_i)_{i\in d}\in \mathcal{S}_n((U_i)_{i\in d})$ by a translation of $(X_i)_{i\in d}$ we mean another $(Y_i)_{i\in d}\in \mathcal{S}_n((U_i)_{i\in d})$ such that $Y_i$ is a translate of $X_i$ for at least one $i\in d$.

         In the inductive step of the proof of Theorem $7$,we are going to consider translations of uniform families defined on $U(t)$ at $U(s)$, for $s,t\in U$ with $s\neq t$. That is why we consider only horizontal translations of trees.\\
          
          We extend now the notion of agreement of Definition $10$ on node-level sets as follows:

    \begin{definition}\label{nodeslevelsrelation}
    Given a finite node set $N \subset U$ we say that two finite strong subtrees $X,Y$ of $U$ \emph{agree} on $N$ if $N\subseteq X$ and $N \subseteq Y$ up to translation, i.e. either $N\subseteq X,Y$ or $N \subseteq X$ and $N' \subseteq Y$, where $N'$ is a translate of $N$. We denote that $X,Y$ \emph{agree} on $N$ by $X:N=Y:N$.

    For a finite level set $L$ now, we say that $X\in \mathcal{S}_n(U),Y\in \mathcal{S}_{n'}(U)$ \emph{agree} on $L$, if for every $m\in L$ we have $X(k),Y(k')\subseteq U(m)$, for some $k\in n$ and some $k'\in n'$. We denote that $X,Y$ agree on $L$ by $X:L=Y:L$. 
    
    Given now a node-level set $(N,L)$ where $L_N<L$, we say that $X,Y$ \emph{agree on} $(N,L)$, if they agree on $N$ and on $L$. We denote that $X,Y$ agree on $(N,L)$ by $X:(N,L)=Y:(N,L)$.
    
    Similarly $(X_i)_{i\in d}$ and $(Y_i)_{i\in d}$, finite strong subtrees of $(U_i)_{i\in d}$, $\mathrm{agree}$ on $(N_i, L_i)_{i\in d}$ if $X_i, Y_i$ agree on $(N_i, L_i)$ for every $i\in d$.
    \end{definition}

   To demonstrate how Definition $10$ and $15$ relate we consider $X'\in \mathcal{C}^U_{(N,L)}$ and $Y'\in\mathcal{C}^U_{(N',L)}$, both of height $n$. Definition $10$ says that $X'$ and $Y'$ agree on $(N,L)$, $N\subseteq b^n, L\subseteq n$, if and only if $N=N'$, $\iota_{b^n,X'}\circ \iota^{-1}_{b^n,Y'}$ is the identity on $N$ and if for every
 $l\in L$ the $l$th level of $X'$ and the $l$th level of $Y'$ both lie on the same level of $U$. Definition $15$ says that $X' \in \mathcal{C}^U_{(N,L)}$ and $Y'\in \mathcal{C}^U_{(N',L)}$ agree on $(N,L)$ if and only if $\iota^{-1}_{b^n,X'}(N)=\iota^{-1}_{b^n,Y'}(N')$ and for every
 $l\in L$ the $l$th level of $X'$ and the $l$th level of $Y'$ both lie on the same level of $U$. Therefore, it allows the node set $N$ to be translated. It allows also agreement between finite strong subtrees of different height.

    For a strong subtree  $X\in \mathcal{C}_{(N,L)}^U$, we define $X^{in}\sqsubseteq X$ as follows: If the node-level set $(N,L)$ is a node set, i.e. $L=\emptyset$, then $X^{in}=X$. If both $N\neq \emptyset$ and $L\neq \emptyset$, then by $X^{in}$ we denote the initial segment of $X$ that covers the node set $N$ and as a result $ N^\wedge$.
   Consider the case of the very first level $l_0$ of the level set $L=\{l_0,\dots ,l_{m}\,\}$ being as $l_0=\max L_N+1$. Notice in this case we cannot choose the successors $N'$ of the nodes in $N$ that lie on $l_0-1$. They get imposed to us by the choice of $l_0$. This pair gives rise to the same strong subtree envelope as the pair with node set $N\cup N'$ and level set $L'=\{ l_1,\dots, l_m\}$.  Therefore we can assume from now on that in any node-level set the level set lies further from the node set.
   Finally if the node-level set is only a level set $(L)$, by $X^{in}$ we denote the initial segment of $X$ whose level set forms an initial segment of $L_U$ i.e. $L_{X^{in}}\sqsubset L_{U}$ and as a result $X^{in}$ forms an initial segment of $U$. In this case $| \{ Y: Y=X^{in}, X\in \mathcal{C}^U_L\}|=1$. If there is not a subset $L_{X^{in}}$ of $L_X$ so that $L_{X^{in}}\sqsubseteq L_U$, then $X^{in}$ is not defined.
  
    In other words $X^{in}\sqsubseteq X$ is the finite strong subtree of $U$ that is a cover of the set of nodes that are in any member of the envelope $\mathcal{C}_{(N,L)}^U$ such that $X\in \mathcal{C}_{(N,L)}^U$. Therefore if we eliminate one node from that set, on any of the resulting strong subtrees $T$ of $U$ it holds that $\mathcal{C}_{(N,L)}^T=\emptyset$.

 Consider now the $d$-sequence $(X_i)_{i\in d}\in \mathcal{C}^{(U_i)_{i\in d}}_{(N_i,L_i)_{i\in d}}$ of strong subtree of $(U_i)_{i\in d}$.
 Notice that it might not be the case that $L_{\cup N_i} < \cup L_i$. Then let $$L_{in}=\{l\in \cup L_i: l \leq \max L_{\cup N_i}\}.$$

 The strong subtree envelope $\mathcal{C}^{U_j}_{(N_i,L_i)_{i\in d}}$ in a fixed coordinate $j\in d$, is defined as  the strong subtree envelope of the set of nodes $ N_j\subset U_j$ and the set of levels $$L^j=\cup_{i\in d}L_i \bigcup_{i\in d, i\neq j}\{L_{{N_i}^\wedge}\}.$$ Then we set $$L^j_{in}=\{ l\in L^j: l\leq \max L_{N_j}\}.$$ Let $n=| L_{ N_j^{\wedge}}\cup  L^j|$ and $\sigma:L_{ N_j^{\wedge}}\cup  L^j\to n$ is the increasing bijection witnessing that $n=|L_{ N_j^{\wedge}}\cup  L^j|$. We define the strong subtree envelope $\mathcal{C}^{U_j}_{(N_i,L_i)_{i\in d}}$ as follows:\\  $\mathcal{C}^{U_j}_{(N_i,L_i)_{i\in d}}=\{\, Y: Y\in \mathcal{S}_n{(U_j)}\text{, } N_j^{\wedge}\subseteq Y \text{ and for every }k\in L^j\text{ with }\sigma(k)=k'\text{,  }Y(k')\subset U_j(k)\,\}.$

 Then the strong subtree envelop of $(N_i,L_i)_{i\in d}$ in $(U_i)_{i\in d}$ as defined above, has another equivalent formulation: $$\mathcal{C}^{(U_i)_{i\in d}}_{(N_i,L_i)_{i\in d}}=\{ (X_i)_{i\in d}: X_j\in \mathcal{C}^{U_j}_{(N_i,L_i)_{i\in d}}\text{ for } j\in d \}$$

  In this case now, for $X_j\in \mathcal{C}^{U_j}_{(N_i,L_i)_{i\in d}}$, $j\in d$ fixed, we define $X_j^{in}\sqsubseteq X_j$ its initial segment that covers $N_j\cup L^j_{in}$, if it is defined. Set
 \begin{equation}
   n=\max \{ \, \vline X_j^{in}\vline: \,j\in d\, \}.
   \end{equation}
             Then define the initial segment $((X_i)_{i\in d})^{in}= (Z_i)_{i\in d}$, of $(X_i)_{i\in d}$ so that the height of $(Z_i)_{i\in d}$ is $n$ and for all $i\in d$ we have $Z_i\sqsubseteq X_i$. Notice that the only possibility of $((X_i)_{i\in d})^{in}$ not being defined is the case that $\bigcup_{i\in d} N_i=\emptyset$ and $L^j=\cup_{i\in d}L_i$ does not contain an initial segment of $L_U$.

         \begin{center}
        \section{Main theorem}
        \end{center}
        
        To state our main theorem we need the following definition:\\
        \begin{definition}\label{canonical}
        A mapping $c$ defined on a uniform family $\mathcal{G}$ of finite strong subtrees on $U$ is called a $\mathrm{ canonical}$ coloring of $\mathcal{G}$ on $U$ if there exists a family of node-level sets on $U$ denoted by $\mathcal{T}$ and a mapping $f:\mathcal{G} \to \mathcal{T}$ such that: 
        \begin{enumerate}
        \item{}   For every $X\in \mathcal{G}$ if $f(X)=(N^X,L^X)$ then $N^X\subseteq X$, $L^X\subseteq L_X$ and $L_{N^X}< L^X$.
        \item{} For any $X,Y\in \mathcal{G}$, $c(X)=c(Y)$ if and only if $f(X)=f(Y)$ up to translation of the node set. 
        \end{enumerate}
        The second condition is equivalent to the existence of a one-to-one, up to translation, mapping $\phi$ defined on $\mathcal{T}$ such that $\phi(f(X))=c(X)$ for all $X\in \mathcal{G}$.\\
        
        Similarly for the case of a $d$-sequence of $b$-branching trees $(U_i)_{i\in d}$.
                 A mapping $c$ defined on a uniform family $\mathcal{G}$ of finite strong subtrees on $(U_i)_{i\in d}$ is called a $\mathrm{canonical}$ coloring of $\mathcal{G}$ on $(U_i)_{i\in d}$ if there exists a family of $d$-sequences of node-level sets on $(U_i)_{i\in d}$ denoted by $\mathcal{T}$ and a mapping $f:\mathcal{G} \to \mathcal{T}$ such that: 
                 \begin{enumerate}
        \item{} For every $(X_i)_{i\in d} \in \mathcal{G}$ if $f((X_i)_{i\in d})=(N^{X_i},L^{X_i})_{i\in d}$ then $N^{X_i}\subseteq X_i$, $L^{X_i}\subseteq L_{X_i}$ and $L_{N^{X_i}}< L^{X_i}$ for all $i\in d$.\\
        \item{} For any $(X_i)_{i\in d},(Y_i)_{i\in d}\in \mathcal{G}$, $c((X_i)_{i\in d})=c((Y_i)_{i\in d})$ if and only if $f((X_i)_{i\in d})=f((Y_i)_{i\in d})$ up to translation of node set. 
        \end{enumerate}
        
        The second condition is equivalent to the existence of a one-to-one, up to translation, mapping $\phi$ defined on $\mathcal{T}$ such that $\phi(f((X_i)_{i\in d})=c((X_i)_{i\in d}))$ for all $(X_i)_{i\in d}\in \mathcal{G}$.\\
         \end{definition}

         In other words two finite strong subtrees $X,Y$ of $U$ get mapped in the same place by $c$ if and only if they agree on a node-level set $(N,L)\in \mathcal{T}$ in the sense of Definition $15$, i.e. $$X:(N,L)=Y:(N,L)$$

           \begin{rem} We must remark that if we take the union of the strong subtree envelopes of all the node-level sets in $\mathcal{T}$ and by passing to a strong subtree, if necessary, we get another uniform family of finite strong subtrees. That new uniform family has rank less than or equal to the rank of $\mathcal{G}$. For a proof see at the very end of this section, Proposition $3$.
        \end{rem}
        The main theorem of this paper is the following:\\
        
        \begin{theorem}\label{maintheorem} For any uniform family of finite strong subtrees $\mathcal{G}$ on $U$, and every mapping $c$ on $\mathcal{G}$, there exists $T\in \mathcal{S}_{\infty}(U)$ such that $c\upharpoonright (\mathcal{G} \upharpoonright T)$ is a canonical coloring of $ \mathcal{G}\upharpoonright T$ on $T$.
       \end{theorem}
       Moreover we have also its version for finite sequences of trees:

        \begin{theorem} For any uniform family of finite strong subtrees $\mathcal{G}$ on $(U_i)_{i\in d}$, and every mapping $c$ on $\mathcal{G}$, there exists $(T_i)_{i\in d}\in \mathcal{S}_{\infty}((U_i)_{i\in d})$ such that $c\upharpoonright ( \mathcal{G}\upharpoonright (T_i)_{i\in d})$ is a canonical coloring of $ \mathcal{G}\upharpoonright (T_i)_{i\in d}$ on $(T_i)_{i\in d}$.\end{theorem}

        Notice that the range of $c$ in both of the above theorems is at most countably infinite. The proofs of Theorems $7$  and $8$ are done by induction on the rank of the uniform family. The case of a $0$-uniform family $\mathcal{G}$ is trivially true. Now assuming that Theorems $7$ and $8$ hold for any $\beta$-uniform family of finite strong subtrees, where $\beta<\alpha$, we are going to show that they both hold for any $\alpha$-uniform family $\mathcal{G}$ on $U$ and any $\alpha$-uniform family $\mathcal{G}$ on $(U_i)_{i\in d}$ respectively. For the inductive step we need to establish some new results. Up to Section $6.1$ we develop the tools that we need in order to do our inductive step.\\
            
        Let  us consider an $\alpha$-uniform family $\mathcal{G}$ on $U$  and an equivalence relation $c$ on it, or equivalently a mapping. By definition $\mathcal{G}(t)$ is a $\beta$-uniform family on $U(t)$, for some $\beta<\alpha$. The inductive hypothesis applies for $c_t$ on $\mathcal{G}(t)$ defined by $c_t((X_i)_{i\in b})=c(t^\frown (X_i)_{i\in b})$ to give us a $U'_t\in \mathcal{S}_{\infty}(U(t))$, $U'_t(0)=t$, where the restriction $c_t\upharpoonright (\mathcal{G}(t)\upharpoonright U'_t(t))$ is a canonical coloring of $\mathcal{G}(t)\upharpoonright U'_t(t)$ on $U'_t(t)$.

        By a simple fusion sequence we get a $T\in \mathcal{S}_{\infty}(U)$ such that for every $t\in T$ the restriction $c_t$ of $c$ on $\mathcal{G}(t)\upharpoonright T(t)$ defined by $c_t((X_i)_{i\in b})=c(t^\frown (X_i)_{i\in b})$ is canonical on $T(t)$. To see that consider $t_0\in U(1)$. By the inductive hypothesis we get $U'_{t_0}\in \mathcal{S}_{\infty}(U[t_0])$, $U'_{t_0}(0)=t_0$, where $c_{t_0}\upharpoonright (\mathcal{G}(t_0)\upharpoonright U'_{t_0}(t_0))$ is a canonical coloring of $\mathcal{G}(t_0)\upharpoonright U'_{t_0}(t_0)$ on $U'_{t_0}(t_0)$. Consider the level set $L_{U'_{t_0}}$. Proceed in $t_1\in U(1)$, let $U''_{t_1}\in \mathcal{S}_{\infty}(U[t_1])$be such that $U''_{t_1}(0)=t_1$, $L_{U''_{t_1}}=L_{U'_{t_0}}$. By the inductive hypothesis we get a $U'_{t_1}\in \mathcal{S}_{\infty}(U''_{t_1})$, $U'_{t_1}(0)=t_1$ where the restriction $c_{t_1}$ is a canonical coloring of $\mathcal{G}(t_1)\upharpoonright U'_{t_1}(t_1)$ on $U'_{t_1}(t_1)$. Repeat that for all nodes $t_i\in U(1)$, $i\in b$. Consider $L_{U_{t_{b-1}}}$. Let $U_{t_i}\in \mathcal{S}_{\infty}(U'_{t_i})$ so that $U_{t_i}(0)=t_i$, $L_{U_{t_i}}=L_{U_{t_{b-1}}}$, for all $i\in b-1$. Set $T(0)=U(0)$, $T(1)=U(1)$ and $T(2)=\bigcup_{i\in b}U_{t_i}(1)$. Suppose we have constructed $T(n)$ and we would like to choose $T(n+1)$. Let $(s_i)_{i\in b^n}$ be an enumeration of the nodes in $T(n)$. Start with $s_0$. By the inductive hypothesis we get $U'_{s_0}\in \mathcal{S}_{\infty}(U[s_0])$, $U'_{s_0}(0)=s_0$ where $c_{s_0}\upharpoonright (\mathcal{G}(s_0)\upharpoonright U'_{s_0}(s_0))$ is a canonical coloring of $\mathcal{G}(s_0)\upharpoonright U'_{s_0}(s_0)$ on $U'_{s_0}(s_0)$. Consider the level set $L_{U'_{s_0}}$. Proceed in $s_1\in T(n)$, let $U''_{s_1}\in \mathcal{S}_{\infty}(U[s_1])$, $U''_{s_1}(0)=s_1$ be such that $L_{U''_{s_1}}=L_{U'_{s_0}}$. By the inductive hypothesis we get a $U'_{s_1}\in \mathcal{S}_{\infty}(U''_{s_1})$, $U'_{s_1}(0)=s_1$ where the restriction $c_{s_1}$ is a canonical coloring of $\mathcal{G}(s_1)\upharpoonright U'_{s_1}(s_1)$ on $U'_{s_1}(s_1)$. Repeat that for all nodes $s_i\in T(n)$, $i\in b^n$. Consider $L_{U_{s_{b^n-1}}}$. Let $U_{s_i}\in \mathcal{S}_{\infty}(U'_{s_i})$ so that $U_{s_i}(0)=s_i$, $L_{U_{s_i}}=L_{U_{t_{b^n-1}}}$ for all $i\in b^n-1$. Set $T(n+1)=\bigcup_{i\in b^n}U_{s_i}(1)$. The limit of this fusion sequence $T\in \mathcal{S}_{\infty}(U)$ has the property that for every $t\in T$ the restriction $c_t$ of $c$ on $\mathcal{G}(t)\upharpoonright T(t)$ is a canonical coloring of $\mathcal{G}(t)\upharpoonright T(t)$ on $T(t)$. For notational simplicity we assume that $T=U$.

       Therefore we have that at each node $t$ of $U$ the restriction $c_t$ of $c$ on $\mathcal{G}(t)\upharpoonright U(t)$, defined by $c_t((X_i)_{i\in b})=c(t^\frown (X_i)_{i\in b})$, is canonical. As a result there exists a family of $b$-sequences of node-level sets, like $(N_i,L_i)_{i\in b}$, denoted by $\mathcal{T}^t$ and a mapping $f_t$ that satisfy conditions $(1)$ and $(2)$ of Definition $16$. The family $\mathcal{T}^t$, by the Remark $2$ above, gives rise to a $\gamma$-uniform family $\mathcal{F}(\mathcal{G})(t)$ on a strong subtree of $U(t)$. By a simple fusion sequence identical with the one just above, we can assume that $\mathcal{F}(\mathcal{G})(t)$ is defined on $U(t)$ for every $t\in U$. The mappings $f_t$ are defined on $\mathcal{G}(t)\upharpoonright U(t)$ and the one-to-one mappings $\phi_t$ are defined on $\mathcal{T}^t$ by $$\phi_t ((N_i, L_i)_{i\in b}=f_t((X_i)_{i\in b}))=c_t( (X_i)_{i\in b})$$ where $\mathcal{C}^{U(t)}_{(N_i,L_i)_{i\in d}}\subset \mathcal{F}(\mathcal{G})(t)$ and $t^\frown (X_i)_{i\in b}\in \mathcal{G}$.

    In that way we can think of $\mathcal{F}$ as a functor defined on the set of all pairs $(\mathcal{G},c)$ of a uniform family of finite strong subtrees on a tree $U$ with a fixed branching number and an equivalence relation $c$ on that family. For every $t\in U$, $\mathcal{F}(\mathcal{G})(t)$ is a uniform family on a strong subtree of $U(t)$ with rank less than or equal to that of $\mathcal{G}(t)$. By $\mathcal{F}(\mathcal{G})$ we denote the uniform family that results from the union of $t^\frown \mathcal{F}(\mathcal{G})(t)$, for all nodes $t$ of $U$.

        From now on we work with the uniform family $\mathcal{F}(\mathcal{G})$ and not with the original uniform family $\mathcal{G}$ that we started with. So all the definitions and notation developed so far apply to the resulting uniform family $\mathcal{F}(\mathcal{G})$. For simplicity reasons from this point up to the end of the paper, we will assume that $\mathcal{F}(\mathcal{G})$ is directly defined on $U$ instead of one of its infinite strong subtrees. As a consequence, $\mathcal{F}(\mathcal{G})(t)$ is assumed to be defined directly on $U(t)$, for all $t\in U$.
        In particular we consider the pair $(\mathcal{F}(\mathcal{G}), c')$ with $c'$ defined on $\mathcal{F}(\mathcal{G})$ by $c'(t^\frown(Y_i)_{i\in d})=\phi_t((N_i,L_i)_{i\in d})$, where $(Y_i)_{i\in d}\in \mathcal{C}^{U(t)}_{(N_i,L_i)_{i\in d}} \subset \mathcal{F}(G)(t)$ and $(N_i,L_i)_{i\in d}=f_t((X_i)_{i\in d})$ for a $(X_i)_{i\in d}\in \mathcal{G}(t)$, $t\in U$. We make identical assumptions in the case of $(U_i)_{i\in d}$.
        
        The last thing to notice is that given any mapping on the $n$-uniform family $\mathcal{S}_n((U_i)_{i\in d})$, by the inductive hypothesis of Theorem $8$, we can assume that the mapping is canonical. There is a family of node-level sets $\mathcal{T}$ that satisfies conditions $(1)$ and $(2)$ of the Definition $16$ and a mapping $f$. Consider the mapping $c^\star: \mathcal{S}_n((U_i)_{i\in d})\to n$ defined by $c^\star((X_i)_{i\in d})=i$ if $\mathcal{C}^{(U_i)_{i\in d}}_{f((X_i)_{i\in d})}$ contains strong subtrees of height equal to $i\in n$. By Theorem $4$ we get a strong subtree $(V_i)_{i\in d}\in \mathcal{S}_{\infty}((U_i)_{i\in d})$ on which $c^\star$ is constant and equal to some fixed $i_0$. Let $k$ be the cardinality of the set of node-level sets $\{ (N^j_i,L^j_i)_{i\in d, j\in k}\}$ such that for any $(Y_i)_{i\in d}\in \mathcal{C}^{(V_i)_{i\in d}}_{(N^j_i,L^j_i)_{i}}$ we have that its height is equal to $i_0$. 
        
        Consider the coloring $\tilde{c}:\mathcal{S}_n((V_i)_{i\in d})\to k$ defined by $\tilde{c}((X_i)_{i\in d})=j\in k$ if and only if $f((X_i)_{i\in d})=(N^j_i,L^j_i)_{i\in d}$. By an application of Theorem $4$ we get a $(V'_i)_{i\in d}\in \mathcal{S}_{\infty}((V_i)_{i\in d})$, so that $\tilde{c}\upharpoonright (V'_i)_{i\in d}$ is constant. Therefore we can assume that for any two node-level sets $(N_i, L_i)_{i\in d}, (N'_i, L'_i)_{i\in d}$ and any two members of their strong subtree envelopes $(X_i)_{i\in d}\in \mathcal{C}^{(U_i)_{i\in d}}_{(N_i, L_i)_{i\in d}}$ and $(Y_i)_{i\in d}\in \mathcal{C}^{(U_i)_{i\in d}}_{(N'_i, L'_i)_{i\in d}}$ one has: $\iota_{b^{i_0},X_i} \circ \iota^{-1}_{b^{i_0},Y_i}(N'_i)=(N_i)$ and $| L_i|= | L'_i |$ for all $i\in d$. Therefore any two members of $\mathcal{F}(\mathcal{G})$ are isomorphic in the sense of Definition $4$.

        We need to obtain some results that they are going to give us the inductive step. The first thing we notice is the following lemma:

    \begin{lemma}  Let $d, d'\in \omega$, $\mathcal{G}$ an $\alpha$-uniform family on $(U_i)_{i\in d}$ and $\lambda:\mathcal{G} \to \bigcup_{j\in d'}F_j$, where $F_j\neq U_i$ for all $i\in d$, $j\in d'$ are also $b$-branching trees of infinite length. There exists for all $i\in d$, $T_i\in\mathcal{S}_{\infty}(U_i)$,  and for all $j\in d'$, $ V_j\in \mathcal{S}_{\infty}(F_j)$, all having the same level sets, such that $$\lambda(\mathcal{G}\upharpoonright (T_i)_{i\in d})\bigcap (\cup_{j\in d'}V_j)=\emptyset$$
    \end{lemma}
   \begin{proof} 
   We are giving a proof by induction on the rank of $\mathcal{G}$. The case of a $0$-uniform family is vacuously true. Consider a $1$-uniform family $\mathcal{G}$ and a mapping $\lambda: \mathcal{G}\to \bigcup_{j\in d'}F_j$. By the inductive hypothesis of Theorem $8$ we can assume that $\lambda$ is canonical i.e. there exists a family $\mathcal{T}$ of $d$-sequences of node-level sets and a one-to-one mapping $\phi$ on $\mathcal{T}$. That family $\mathcal{T}$ gives rise to a uniform family $\mathcal{F}(\mathcal{G})$ on $(U_i)_{i\in d}$. If $\mathcal{T}=\emptyset$, so the rank of $\mathcal{F}(\mathcal{G})$ is zero, then the mapping $\lambda$ is constant and the assertion of our lemma is trivial. Let $\mathcal{T}\neq \emptyset$. By Remark $2$ observe that the rank of $\mathcal{F}(\mathcal{G})$ is equal to one because the rank of $\mathcal{G}$ is equal to one. As a result the set $\mathcal{T}$ contains $d$-sequences of either node or level sets, if otherwise by taking the strong subtree envelop of $(N_i,L_i)_{i\in d}\in \mathcal{T}$ we would get finite strong subtrees of height greater than $1$ contradicting that the rank of $\mathcal{F}(\mathcal{G})$ is equal to one. Therefore for $(N_i,L_i)_{i\in d}\in \mathcal{T}$ we have that either $N_i=\emptyset$ or $L_i=\emptyset$, for all $i\in d$.
   Pick strong subtrees $(X^1_i)_{i\in d}\in \mathcal{S}_1((U_i)_{i\in d})$and $(Y^1_j)_{j\in d'}\in \mathcal{S}_1((V_j)_{j\in d'})$ so that $L_{(X^1_i)_{i\in d}}=L_{(Y^1_j)_{j\in d'}}=n\in L_{(U_i)_{i\in d}}=\omega$ and such that: $$\lambda((X^1_i)_{i\in d})\notin \cup_{j\in d'}Y^1_j.$$ 
  
     For every $t\in \bigcup_{j\in d'}Y_j$, look at the level set, if non empty, of $\lambda^{-1}(t)$. Then for each such a $t$ subtract the level $L_{\lambda^{-1}(t)}$ from both level sets $L_{(U_i)_{i\in d}}$ and $L_{(V_j)_{j\in d'}}$. Having done that for all $t\in \bigcup_{j\in d'}Y^1_j$ we get strong subtrees $(T^1_i)_{i\in d}\sqsupseteq (X^1_i)_{i\in d}$ and $(V^1_j)_{j\in d'}\sqsupseteq(Y^1_j)_{j\in d'}$ with the same levels sets. To be precise $L_{(T^1_i)_{i\in d}}=L_{(V_j)_{j\in d'}}=L_{(U_i)_{i\in d}}\setminus \{L_{\lambda^{-1}(t)}: t\in \bigcup_{j\in d'}Y_j\}$.
     These two strong subtrees have the property that for any $(Z_i)_{i\in d}\in \mathcal{S}_1((T^1_i)_{i\in d})$, $\lambda((Z_i)_{i\in d})\notin \bigcup_{j\in d'}Y^1_j$. To see that notice that for any $t\in \bigcup_{j\in d'} Y^{1}_j$ if there exists $(Z_i)_{i\in d}\in \mathcal{C}^{(U_i)_{i\in d}}_{(N_i,L_i)_i}$ so that $\lambda((Z_i)_{i\in d})=t$, then $\mathcal{C}^{(T^{1}_i)_{i\in d}}_{(N_i,L_i)_i}=\emptyset$. This is because we have removed the level $L_{\lambda^{-1}(t)}=L_{(N_i,L_i)_i}$.
     
     Set $$(T_i)_{i\in d}\upharpoonright 1=(T^1_i)_{i\in d}\upharpoonright 1=(X^1_i)_{i\in d}\text{ and }(V_j)_{j\in d'}\upharpoonright 1=(V^1_j)_{j\in d'}\upharpoonright 1=(Y_j^1)_{i\in d'}.$$
     
         Suppose we have chosen the restrictions $(T_i)_{i\in d}\upharpoonright n=(X^{n}_i)_{i\in d}\sqsubseteq (T^n_i)_{i\in d}$ and $(V_j)_{j\in d'}\upharpoonright n=(Y_j^{n})_{j\in d'}\sqsubseteq (V^n_j)_{j\in d'}$. We would like to decide the $(T_i)_{i\in d}\upharpoonright n+1$ and $(V_j)_{j\in d'}\upharpoonright n+1$. Then pick a level $m' \in L_{(V^n_j)_{j\in d'}}$ such that the successors of each node in $Y^n_j(n-1)$ on $V^n_j(m')$ are more than $b^{n\cdot d}$. Now for any choice of successors $\bigcup_{i\in d} X^{n}_i(n-1)$ on $\cup_{i\in d} T^n_i(m')$ we can always choose successors of $\bigcup_{j\in d'} Y^{n}_j(n-1)$, that lie on $\bigcup_{j\in d'}V^n_j(m')$, so that the resulting strong subtrees $(X^{n+1}_i)_{i\in d}$ and $(Y^{n+1}_j)_{j\in d'}$, both of length $n+1$, satisfy: $\lambda((X'_i)_{i\in d})\notin \bigcup_{j\in d'} Y^{n+1}_j(n)$, for all $(X'_i)_{i\in d}\in \mathcal{S}_1(X^{n+1}_i)_{i\in d}$. For any $t\in \bigcup_{j\in d'} Y^{n+1}_j(n)$ subtract the level $L_{\lambda^{-1}(t)}$ from both level sets $L_{(T^{n}_i[X^{n+1}_i])_{i\in d}}$ and $L_{(V^{n}_j[Y^{n+1}_j])_{j\in d'}}$. Having done that for all $t\in \bigcup_{j\in d'} Y^{n+1}_j(n)$ we get strong subtrees $(T^{n+1}_i)_{i\in d}\in \mathcal{S}_{\infty}((T^{n}_i[X^{n+1}_i])_{i\in d})$ and $(V^{n+1}_j)_{j\in d'}\in \mathcal{S}_{\infty}((V^{n}_j[Y^{n+1}_j])_{j\in d'})$ such that $(X^{n+1}_i)_{i\in d}\sqsubseteq (T^{n+1}_i)_{i\in d}$ and $(Y^{n+1}_j)_{j\in d'}\sqsubseteq (V^{n+1}_j)_{j\in d'}$. These strong subtrees satisfy that for any $(X_i)_{i\in d}\in \mathcal{S}_1((T^{n+1}_i)_{i\in d})$ it holds that $\lambda((X_i)_{i\in d})\cap (\cup_{j\in d'} Y^{n+1}_j)=\emptyset$. To see that notice that for any $t\in \bigcup_{j\in d'} Y^{n+1}_j$ if there exists $(X_i)_{i\in d}\in \mathcal{C}^{(T^{n}_i)_{i\in d}}_{(N_i,L_i)_i}$ so that $\lambda((X_i)_{i\in d})=t$, then $\mathcal{C}^{(T^{n+1}_i[X^{n+1}_i])_{i\in d}}_{(N_i,L_i)_i}=\emptyset$. This is cause we have removed the level $L_{\lambda^{-1}(t)}=L_{(N_i,L_i)_i}$.

      Set $$(T_i)_{i\in d}\upharpoonright n+1=(X^{n+1}_i)_{i\in d}\text{ and }(V_j)_{j\in d'}\upharpoonright n+1=(Y_j^{n+1})_{i\in d'}.$$
      Let $(T_i)_{i\in d}$ be such that $(T_i)_{i\in d}\upharpoonright n=(X^{n}_i)_{i\in d}\text{ and }(V_j)_{j\in d'}\upharpoonright n=(Y_j^{n})_{i\in d'}$ for all $n\in \omega$. $(T_i)_{i\in d}$ and $(V_j)_{j\in d'}$ satisfy the conclusions of our lemma. Suppose not, let $(X_i)_{i\in d}\in \mathcal{S}_1((T_i)_{i \in d})$, $s\in \bigcup_{j\in d'} V_j$ with $|s|=k$, be so that $\lambda((X_i)_{i\in d})=s$. Then $s\in (Y_j^{k+1})_{j\in d'}$. By construction we have that $\lambda((X_i)_{i\in d})\cap(\bigcup_{j\in d'} Y_j^{k+1})=\emptyset$, a contradiction.
      
  So far we have shown that the statement of our lemma holds in the case of a uniform family of rank $0$ and of rank $1$. 
  Assume now that our lemma holds for any $\beta$-uniform family, $\beta<\alpha$ and consider an $\alpha$-uniform family $\mathcal{G}$ on $(U_i)_{i\in d}$. Pick an arbitrary $t=(t_0,\dots, t_{d-1})\in \prod_{i\in d} U_i(n)$, for some $n$, and $s=(s_0,\dots ,s_{d'-1})\in \prod_{i\in d'} V_i(n)$. By definition $\mathcal{G}(t)$ is a $\beta$-uniform family, $\beta<\alpha$, on $(U_i)_{i\in d}(t)$, a $d\cdot b$ sequence of trees. The inductive hypothesis applies on $$\lambda_t:(U_i)_{i\in d}(t)\to\bigcup_{i\in d'} F_i(s_i)$$ defined by $$\lambda_t((X_k)_{k\in d\cdot b})=\lambda (t^\frown (X_k)_{k\in d\cdot b})$$ to give us strong subtrees $(T^1_k)_{k\in d\cdot b}$ and $(V^1_m)_{m\in d' \cdot b}$ that satisfy\\ $\lambda (t^\frown (X_k)_{k\in d\cdot b})\notin \bigcup_{m\in d'\cdot b}V^1_m$, for all $(X_k)_{k\in d\cdot b}\in \mathcal{G}(t)\upharpoonright (T^1_k)_{k\in d\cdot b}$. Set $$(T^2_i)_{i\in d}=t^\frown (T^1_i)_{i\in d}\text{ and }(V^2_j)_{j\in d'}=s^\frown(V^1_j)_{j\in d'},$$
   and $$(T_i)_{i\in d}\upharpoonright 2=(T^2_i)_{i\in d}\upharpoonright 2 \text{ and }(V_j)_{j\in d'}\upharpoonright 2=(V^2_j)_{j\in d'}\upharpoonright 2.$$

   We can assume that $\{s_0,\dots ,s_{d'-1}\}\cap \lambda (t^\frown (T^1_k)_{k\in d\cdot b})=\emptyset$. To see that consider the level set of $\lambda_t^{-1}(s_j)$ and subtract a level $l_{s_j}$ in $L_{\lambda_t^{-1}(s_j)}$ from both level sets $L_{(U_i(t_i))_{i\in d}}$ and $L_{(F_j(s_j))_{j\in d'}}$. Having done that for all $s_j$, $j\in d'$ we get strong subtrees $(T'^1_i)_{i\in d}\sqsupseteq(t_0,\dots, t_{d-1})=t$ and $(V'^1_j)_{j\in d'}\sqsupseteq(s_0,\dots ,s_{d'-1})=s$ with the same levels sets. Namely $L_{(T'^1_i)_{i\in d}}=L_{(V'^1_j)_{j\in d'}}=L_{(U_i(t_i))_{i\in d}}\setminus \{l_{s_j}: s_j\in s=(s_0,\dots ,s_{d'-1})\}$.
   These two strong subtrees have the property that for any $(Z_i)_{i\in d\cdot b}\in \mathcal{G}(t)\upharpoonright (T'^1_i)_{i\in d}$, $\lambda_t((Z_i)_{i\in d})\notin \{s_0,\dots, s_{d'-1}\}$. To see that suppose there exist $(Z_i)_{i\in d\cdot b}\in \mathcal{G}(t)\upharpoonright (T'^1_i)_{i\in d}$, $(Z_i)_{i\in d\cdot b}\in \mathcal{C}^{(T'^1_i)_{i\in d}}_{(N_i,L_i)_i}$ and $s_j\in s=(s_0,\dots ,s_{d'-1})$ such that $\lambda_t((Z_i)_{i\in d})=s_j$. There exists $l_{s_j}\in L_{(N_i,L_i)_i}$ so that $l_{s_j}\notin L_{(T'^1_i)_{i\in d}}$. As a result $\mathcal{C}^{(T'^1_i)_{i\in d}}_{(N_i,L_i)_i}=\emptyset$, a contradiction.

 Suppose we have constructed $(T_i)_{i\in d} \upharpoonright n=(T^n_i)_{i\in d}Ê\upharpoonright n$ and $(V_j)_{j\in d'}\upharpoonright n=(V^n_j)_{j\in d'}\upharpoonright n$ so that for any $(X_i)_{i\in d}\in \mathcal{G}(t')$, $t'\in \prod_{i\in d}T^n_i(k)$ for $k<n$ it holds that $\lambda(t'^\frown(X_i)_{i\in d})\cap(\bigcup_{j\in d'}V^n_j\upharpoonright n)=\emptyset$.
We wish to decide $(T_i)_{i\in d}\upharpoonright n+1$ and $(V_j)_{j\in d'}\upharpoonright n+1$.
 
 Let $\{r_0, \dots, r_{d\cdot b^{n-1}-1}\}$ be a one-to-one enumeration of the nodes $\bigcup_{i\in d} T^n_i(n-1)$ and $\{ s'_0, \dots, s'_{d'\cdot b^{n-1}-1}\}$ a one-to-one enumeration of the nodes $\bigcup_{j\in d'} V^n_j(n-1)$.

   For any $r=(r_{k_i})_{i\in d}$, where for all $i\in d$, $r_{k_i}\in T^n_i$, consider the uniform family $\mathcal{G}(r)\upharpoonright (T^n_{i})_{i\in d}(r)$. Apply once more the inductive hypothesis on $(T^n_{i})_{i\in d}(r)$ and $(F^n_j(s'_m)_{m\in [j \cdot b^{n-1}, (j+1)\cdot b^{n-1})})_{j\in d'}$ to get strong subtrees $(T'^n_l)_{l\in d\cdot b}\in \mathcal{S}_{\infty}( (T^n_{i})_{i\in d}(r))$ and $(F'^n_f)_{f\in d'\cdot b^n}\in \mathcal{S}_{\infty} ((F^n_j(s'_m)_{m\in [j \cdot b^{n-1}, (j+1)\cdot b^{n-1})})_{j\in d'})$ that satisfy the conclusions of our lemma.  At this point we can assume that $\{ s'_0, \dots, s'_{d'\cdot b^{n-1}-1}\}\cap \lambda_r(\mathcal{G}(r)\upharpoonright (T^n_{i})_{i\in d}(r))=\emptyset$.  That can be guaranteed by the fact that $\lambda_r$ on $\mathcal{G}(r)\upharpoonright (T^n_{i})_{i\in d}(r)$  is a canonical coloring on a uniform family of rank $\beta< \alpha$. The argument is identical with the one just above.
     Having done that for all possible $r$ as above, we get strong subtrees $$(T^{n+1}_g)_{g\in d\cdot b^{n}}\in \mathcal{S}_{\infty}((T^n_i(r_{k})_{k\in [i\cdot b^{n-1}, (i+1)\cdot b^{n-1}})_{i\in d})$$ and $$(V^{n+1}_f)_{f\in d'\cdot b^{n}}\in \mathcal{S}_{\infty}((V^n_j(s'_{m})_{m\in [j\cdot b^{n-1}, (j+1)\cdot b^{n-1})})_{j\in d'})$$ all with the same level sets. Let $(T^{n+1}_i)_{i\in d}=((T^n_i)_{i\in d}\upharpoonright n)^\frown (T^{n+1}_g)_{g\in d\cdot b^{n}}$ and $(V^{n+1}_j)_{j\in d'}=((V^n_j)_{j\in d'}\upharpoonright n)^\frown (V^{n+1}_f)_{f\in d'\cdot b^{n}}$. Set
   $$(T_i)_{i\in d}\upharpoonright n+1=(T^{n+1}_i)_{i\in d}\upharpoonright n+1\text{ and }(V_j)_{j\in d'}\upharpoonright n+1=(V^{n+1}_j)_{j\in d'}\upharpoonright n+1$$
   
   For any $(X_i)_{i\in d}\in \mathcal{G}(t')$, $t'\in \prod_{i\in d}T^{n+1}_i(k)$ where $k<n+1$, it holds that $\lambda(t'^\frown (X_i)_{i\in d})\cap(\cup_{j\in d'}V^{n+1}_j\upharpoonright n+1)=\emptyset$.
     The resulting strong subtrees $(T_i)_{i\in d}$ such that $(T_i)_{i\in d}\upharpoonright n=(T^{n}_i)_{i\in d}\upharpoonright n\text{ and }(V_j)_{j\in d'}\upharpoonright n=(V_j^{n})_{i\in d}\upharpoonright n$ for all $n\in \omega$, satisfy the conclusions of our lemma, with an argument identical with that of the case of rank equal to $1$.

  \end{proof}

     Having established the previous lemma, we prove the following:   
        
            \begin{lemma} Let $d\in\omega$, $\mathcal{G}$ an $\alpha$-uniform family on $(U_i)_{i\in d}$ and $\lambda: \mathcal{G} \to (U_i)_{i\in d}$ be a mapping with the property that: $\lambda(X_0,\dots, X_{d-1})\notin \bigcup_{i\in d}X_i$, for all $(X_0,\dots,X_{d-1})\in \mathcal{G}$. There exists a strong subtree $(T_i)_{i\in d}\in \mathcal{S}_{\infty}((U_i)_{i\in d})$ such that $$ \lambda(\mathcal{G}\upharpoonright (T_i)_{i\in d})\bigcap (\cup_{i\in d}T_i)=\emptyset$$
    \end{lemma}
       \begin{proof} We give a proof by induction on the rank of $\mathcal{G}$. For a $0$-uniform family the assertion of the lemma is vacuously true. Let $\mathcal{G}$ be a $1$-uniform family and $\lambda:\mathcal{G}\to (U_i)_{i\in d}$ be a mapping with the property that $\lambda((X_i)_{i\in d})\notin \bigcup_{i\in d}X_i$. By the inductive hypothesis of Theorem $8$ we can assume that $\lambda$ is canonical i.e. there exists a non empty family $\mathcal{T}$ of node-level sets, which gives rise to a uniform family $\mathcal{F}(\mathcal{G})$ on $(U_i)_{i\in d}$.   
     
     Pick now $(X^1_i)_{i\in d}\in \mathcal{S}_1((U_i)_{i\in d})$ and for every $t\in \bigcup_{i\in d}X^1_i$ consider the level set, if non empty, $L_{\lambda^{-1}(t)}$. Then subtract for each $t\in \bigcup_{i\in d} X^1_i$ the level $L_{\lambda^{-1}(t)}$ from $L_{(U_i)_{i\in d}}$ so that a resulting strong subtree $(T^1_i)_{i\in d}$ of $(U_i)_{i\in d}$ with $(X^1_i)_{i\in d}\sqsubseteq (T^1_i)_{i\in d}$, has the property that for any $(X'_i)_{i\in d}\in \mathcal{S}_1((T^1_i)_{i\in d})$, $\lambda((X'_i)_{i\in d})\notin \bigcup_{i\in d}X^1_i$. Set $(T_i)_{i\in d}\upharpoonright 1=(T^1_i)_{i\in d}\upharpoonright 1=(X^1)_{i\in d}$. Suppose we have constructed $(T_i)_{i\in d}\upharpoonright n=(T^n_i)_{i\in d}\upharpoonright n=(X^n_i)_{i\in d}$ and we have to decide $(T_i)_{i\in d}\upharpoonright (n+1)$.

    Let $\{ t_0,\dots, t_{b^{n-1}-1}\}$ be a one-to-one enumeration of the nodes $\bigcup_{i\in d} T^n_i(n-1)$. Pick $m \in L_{(T^n_i)_{i\in d}}$ such that any $t\in \{ t_0,\dots, t_{b^{n-1}-1}\}$ has more than $b^{n\cdot d}$ successors on $\bigcup_{i\in d} T^n_i(m)$. Choose successors of $\{ t_0,\dots, t_{b^{n-1}-1}\}$ on $\cup_{i\in d}T^n_i(m)$ so that the resulting strong subtree $(X^{n+1}_i)_{i\in d}$, of length $n+1$, where $(X^{n+1}_i)_{i\in d}\sqsupseteq (X^n_i)_{i\in d}$ has the following property: $\lambda((Z_i)_{i\in d})\notin \cup_{i\in d}X^{n+1}_i(n)$ for any $(Z_i)_{i\in d}\in \mathcal{S}_1((X^{n+1}_i)_{i\in d})$. Consider the strong subtree $(T^n_i[X^{n+1}_i])_{i\in d}$. Now for any $t\in \cup_{i\in d}X^{n+1}_i(n)$ subtract the level $L_{\lambda^{-1}(t)}$ from the level set $L_{(T^n_i[X^{n+1}_i])_{i\in d}}$. Let $(T^{n+1}_i)_{i\in d}$ be a resulting strong subtree with $(X^{n+1}_i)_{i\in d}\sqsubseteq (T^{n+1}_i)_{i\in d}$. For every $(Z_i)_{i\in d}\in \mathcal{S}_1((T^{n+1}_i)_{i\in d})$ we have that $\lambda ((Z_i)_{i\in d})\notin \bigcup_{i\in d}X^{n+1}_i$. Set $$(T_i)_{i\in d}\upharpoonright n+1=(T^{n+1}_i)_{i\in d}\upharpoonright n+1=(X^{n+1}_i)_{i\in d}.$$
    
      Let $(T_i)_{i\in d}$ be such that $(T_i)_{i\in d}\upharpoonright n=(T^{n}_i)_{i\in d}$ for all $n\in \omega$. We claim that it satisfies the conclusions of our lemma. Suppose that $(X_i)_{i\in d}\in \mathcal{S}_1((T_i)_{i\in d})$ and $\lambda((X_i)_{i\in d})=t \in \bigcup_{i\in d} T_i$ with $|t|=k$. By our construction we have that $\lambda((X_i)_{i\in d})\notin \bigcup_{i\in d}T^{k+1}_i\upharpoonright k+1=\bigcup_{i\in d}T_i\upharpoonright k+1$, a contradiction.

      Assume now the lemma holds for any $\beta$-uniform family, $\beta<\alpha$ and consider an $\alpha$-uniform family $\mathcal{G}$ on $(U_i)_{i\in d}$. Pick $t=(t_0,\dots, t_{d-1}) \in \prod_{i\in d} U_i(n)$, for some $n\in \omega$. By definition $\mathcal{G}(t)$ is a $\beta$-uniform family on $(U_i)_{i\in d}(t)$. Apply our assumption to the canonical mapping $\lambda_t: \mathcal{G}(t)\to \bigcup_{i\in d} U_i(t_i)$, defined by $\lambda_t((X_m)_{m\in d\cdot b})=\lambda(t^\frown (X_m)_{m\in d\cdot b})$, to get strong subtrees $(T^1_m)_{m\in d\cdot b} \in \mathcal{S}_{\infty}((U_i(t_i))_{i\in d})$ such that for any $(X_m)_{m\in d\cdot b}\in \mathcal{G}(t)$ one has $\lambda_t ((X_m)_{m\in d\cdot b})\notin \bigcup_{m\in d\cdot b} T^1_m$. We can also assume, as in Lemma $6$ above, that $t=(t_0,\dots ,t_{d-1}) \cap \lambda_t(\mathcal{G}(t)\upharpoonright ( T^1_m)_{m\in d\cdot b})=\emptyset$ since $\lambda_t$ is a coloring on a uniform family or rank $\beta<\alpha$. Set $$(T^2_i)_{i\in d}=t^\frown ( T^1_m)_{m\in d\cdot b}\text{ and }(T_i)_{i\in d}\upharpoonright 2=(T^2_i)_{i\in d}\upharpoonright 2.$$

      Suppose we have constructed $(T_i)_{i\in d}\upharpoonright n=(T^n_i)_{i\in d}\upharpoonright n$ and we have to decide $(T_i)_{i\in d}\upharpoonright (n+1)=(T^{n+1}_i)_{i\in d}\upharpoonright (n+1)$.
     Let $\{r_0,\dots, r_{d\cdot b^{n-1} -1} \}$ be a one-to-one enumeration of the terminal nodes of $(T^n_i)_{i\in d}\upharpoonright n$. Let $r=(r_{k_i})_{i\in d}$, where for all $i\in d$, $r_{k_i}\in T^n_i$. Let also $w=d\cdot b^{n-1}$, $w_r=\{j \in d\cdot b^{n-1}: r_j\in r\}$ and $w_r^{c}=w\setminus w_r$. Consider the mappings $\lambda_r: \mathcal{G}(r)\upharpoonright (T^n_i)_{i\in d}(r)  \to \cup_{j\in w_r^{c}} T^n_i(r_j) $, where $r_j\in T^n_i$, defined by $\lambda_r((X_k)_{k\in w_r })=\lambda(r ^\frown(X_k)_{k\in w_r })$. By the inductive hypothesis we assume that $\lambda_r( \mathcal{G}(r)\upharpoonright (T^n_i)_{i\in d}(r))\bigcap (\cup(T^n_i)_{i\in d}(r))=\emptyset$. Now by Lemma $6$ we get strong subtrees $(T'^n_j)_{j\in w_r}$ of $(T^n_i[r_{k_i}])_{i\in d}$ and $(T'^n_j)_{j\in w^c_r}$ of $(T^n_i[r_j])_{j\in w_r^c}$, all with the same levels sets, that satisfy $$\lambda_r((T'^n_j)_{j\in w_r})\bigcap (\bigcup_{j\in w^c_r}T'^n_j)=\emptyset.$$ At this point we can assume that $\{ r_j: j\in w_r \}\bigcap \lambda_r((T^n_i)_{i\in d}(r))=\emptyset$ by the fact that $\lambda_r$ is a canonical coloring restricted on a uniform family of rank $\beta< \alpha$, as we did in Lemma $6$ above. Repeat this last step for all possible such a $r$ to get strong subtrees $(T'^{n+1}_f)_{f\in d\cdot b^n}\in \mathcal{S}_{\infty}((T'^n_j)_{j\in w})$. Set 
        $$(T^{n+1}_i)_{i\in d}=((T^n_i)_{i\in d}\upharpoonright n)^\frown (T'^{n+1}_f)_{f\in d\cdot b^n}\text{ and }(T_i)_{i\in d}\upharpoonright (n+1)=(T^{n+1}_i)_{i\in d}\upharpoonright (n+1).$$
        
        Let $(T_i)_{i\in d}$ be such that $(T_i)_{i\in d}\upharpoonright n=(T^{n}_i)_{i\in d}$ for all $n\in \omega$. It satisfies the conclusions of our lemma with an argument identical with that in the case of rank equal to one.

    \end{proof}
      
  We would like to establish a result that will give us the possibility of comparing two uniform families and two canonical colorings defined on them.
    We use Lemma $7$ to prove the following:
   
   \begin{lemma} 
   Let $\mathcal{T}_1$ and $\mathcal{T}_2$ be two families of node-level sets so that they generate two uniform families $\mathcal{F}(\mathcal{G}_1)$ and $\mathcal{F}(\mathcal{G}_2)$, on $(U_i)_{i\in d}$, by taking the union of all strong subtree envelopes of all members of $\mathcal{T}_1$ and $\mathcal{T}_2$ respectively . Let $c'_1$ a mapping on $\mathcal{F}(\mathcal{G}_1)$ with the property that $c'_1((X^1_i)_{i\in d})=c'_1((X^2_i)_{i\in d})$ if and only if $(X^1_i)_{i\in d}:(N^1_i, L^1_i)_{i\in d}=(X^2_i)_{i\in d}:(N^1_i, L^1_i)_{i\in d}$ for $(N^1_i, L^1_i)_{i\in d}\in \mathcal{T}_1$. Let also $c_2$ a mapping on $\mathcal{F}(\mathcal{G}_2)$ such that $c_2((Y^1_i)_{i\in d})=c_2((Y^2_i)_{i\in d})$ if and only if $(Y^1_i)_{i\in d}:(N^2_i, L^2_i)_{i\in d}=(Y^2_i)_{i\in d}:(N^2_i, L^2_i)_{i\in d}$ for $(N^2_i, L^2_i)_{i\in d}\in \mathcal{T}_2$. There exists $(T_i)_{i\in d}\in \mathcal{S}_{\infty}((U_i)_{i\in d})$ such that one of the following two statements holds.
   \begin{enumerate}
   \item{} $\mathcal{F}(\mathcal{G}_1)\upharpoonright (T_i)_{i\in d}=\mathcal{F}(\mathcal{G}_2)\upharpoonright (T_i)_{i\in d}$ and $c'_1((X_i)_{i\in d})=c'_2((X_i)_{i\in d})$  for every $(X_i)_{i\in d}\in \mathcal{F}(\mathcal{G}_1)\upharpoonright (T_i)_{i\in d}= \mathcal{F}(\mathcal{G}_2)\upharpoonright (T_i)_{i\in d}$.
   \item{} The image of $c'_1$ on $\mathcal{F}(\mathcal{G}_1)\upharpoonright (T_i)_{i\in d}$ and the image of $c'_2$ on $\mathcal{F}(\mathcal{G}_2)\upharpoonright (T_i)_{i\in d}$ are disjoint.
  
   \end{enumerate}
   \end{lemma}
   \begin{proof}
        
       Partition $\mathcal{F}(\mathcal{G}_1)$ into two pieces $\mathcal{S}_{1,1}$ and $\mathcal{S}_{1,2}$ as follows: $(X_i)_{i\in d}\in \mathcal{S}_{1,1}$ if and only if $(X_i)_{i\in d}\in \mathcal{F}(\mathcal{G}_2), \, c'_1((X_i)_{i\in d})=c'_2((X_i)_{i\in d})$ and $(X_i)_{i\in d}\in \mathcal{S}_{1,2}$ if and only if $ (X_i)_{i\in d} \notin \mathcal{S}_{1,1}$.  Since $\mathcal{F}(\mathcal{G}_1)$ is Ramsey, we get $(T^0_i)_{i\in d}\in \mathcal{S}_{\infty}((U_i)_{i\in d})$ such that either $\mathcal{F}(\mathcal{G}_1)\upharpoonright (T^0_i)_{i\in d}\subseteq \mathcal{S}_{1,1}$, in which case we have the first statement holding, or $\mathcal{F}(\mathcal{G}_1)\upharpoonright (T^0_i)_{i\in d}\subseteq \mathcal{S}_{1,2}$, in which case we have to show that the second statement is on hold.\\ Therefore we assume that $\mathcal{F}(\mathcal{G}_1)\upharpoonright (T^0_i)_{i\in d}\subseteq \mathcal{S}_{1,2}$ and we show that the second statement is true.
      Note that for $(X_i)_{i\in d}\in\mathcal{F}( \mathcal{G}_1)\upharpoonright (T^0_i)_{i\in d}$ to be a member of $\mathcal{S}_{1,2}$ it is either the case that $(X_i)_{i\in d}\notin \mathcal{F}(\mathcal{G}_2)\upharpoonright (T^0_i)_{i\in d}$ or if $(X_i)_{i\in d}\in \mathcal{F}(\mathcal{G}_2)\upharpoonright (T^0_i)_{i\in d}$ then one must have $c'_1 ((X_i)_{i\in d})\neq c'_2((X_i)_{i\in d})$. \\
      
       Let $(X_i)_{i\in d}\in\mathcal{F}(\mathcal{ G}_1)\upharpoonright (T^0_i)_{i\in d}$ and pick, if it exists, a $(Y_i)_{i\in d}\in \mathcal{F}(\mathcal{G}_2)\upharpoonright (T^0_i)_{i\in d}$ such that $$c'_1((X_i)_{i\in d})=c'_2((Y_i)_{i\in d})$$ This would imply that $(X_i)_{i\in d}\neq (Y_i)_{i\in d}$ and that will be true not only for $(X_i)_{i\in d}$, $(Y_i)_{i\in d}$ but for all members of the strong subtree envelope of $(N^1_i,L^1_i)_{i\in d}\in \mathcal{T}_1$, $(N^2_i,L^2_i)_{i\in d}\in \mathcal{T}_2$, where $(X_i)_{i\in d}\in \mathcal{C}^{(T^0_i)_{i\in d}}_{(N^1_i,L^1_i)_{i\in d}}$ and $(Y_i)_{i\in d}\in \mathcal{C}^{(T^0_i)_{i\in d}}_{(N^2_i,L^2_i)_{i\in d}}$. To see this observe that if we had $(X'_i)_{i\in d}=(Y'_i)_{i\in d}$ for some $(X'_i)_{i\in d}\in \mathcal{C}^{(T^0_i)_{i\in d}}_{(N^1_i,L^1_i)_{i\in d}}$ and some $(Y'_i)_{i\in d}\in \mathcal{C}^{(T^0_i)_{i\in d}}_{(N^2_i,L^2_i)_{i\in d}}$ i.e. $(X'_i)_{i\in d} \in \mathcal{F}(\mathcal{G}_2)\upharpoonright (T^0_i)_{i\in d}$, then we would get a contradiction because $c'_1((X'_i)_{i\in d})=c'_1((X_i)_{i\in d})=c'_2((Y_i)_{i\in d})=c'_2((Y'_i)_{i\in d})=c'_2((X'_i)_{i\in d})$ and we have assumed that $\mathcal{F}(\mathcal{G}_1)\upharpoonright (T^0_i)_{i\in d}\subseteq \mathcal{S}_{1,2}$ .\\

      To proceed further, we need the following lemma:
      
      \begin{lemma} In the above context, i.e. $\mathcal{F}(\mathcal{G}_1)\upharpoonright (T^0_i)_{i\in d}\subseteq \mathcal{S}_{1,2}$, by passing to a strong subtree if necessarily, we can assume that there are not $$(X_i)_{i\in d}\in  \mathcal{F}(\mathcal{G}_1)\upharpoonright (T^0_i)_{i\in d}\text{ and }(Y_i)_{i\in d}\in \mathcal{F}(\mathcal{G}_2)\upharpoonright (T^0_i)_{i\in d}$$ such that $c'_1((X_i)_{i\in d})=c'_2((Y_i)_{i\in d})$ and $((X_i)_{i\in d})^{in} =( (Y_i)_{i\in d})^{in}$.
      \end{lemma}
      
      \begin{proof} For simplicity reasons in the proof we write $\mathcal{F}(\mathcal{G}_j)$, $j\in \{1,2\}$ instead of $\mathcal{F}(\mathcal{G}_j)\upharpoonright (T^0_i)_{i\in d}$.
      Suppose now that $c'_1((X_i)_{i\in d})=c'_2((Y_i)_{i\in d})$ for $(X_i)_{i\in d}\in  \mathcal{F}(\mathcal{G}_1)$, $(Y_i)_{i\in d}\in \mathcal{F}(\mathcal{G}_2)$ and $L_{(X_i)_{i\in d}}=L_{(Y_i)_{i\in d}}$. If one has $((X_i)_{i\in d})^{in}=((Y_i)_{i\in d})^{in}=(Z_i)_{i\in d}$, this would imply that there exist $(X'_i)_{i\in d}\in \mathcal{C}^{(T^0_i)_{i\in d}}_{(Z_i,L^1_i)_{i\in d}}$ and $(Y'_i)_{i\in d}\in \mathcal{C}^{(T^0_i)_{i\in d}}_{(Z_i,L^2_i)_{i\in d}}$ so that $X'_i=Y'_i$ for all $i\in d$ and $c'_1((X'_i)_{i\in d})=c'_2((Y'_i)_{i\in d}$ a contradiction.  From now on we consider the case of $L_{(X_i)_{i\in d}}\neq L_{(Y_i)_{i\in d}}$.

   The proof is by induction on the  countable ordinals $\alpha, \beta$ the ranks of $\mathcal{F}(\mathcal{G}_1)$ and $\mathcal{F}(\mathcal{G}_2)$ respectively. Let both $\alpha, \beta$ be finite.  Let $(X_i)_{i\in d}\in \mathcal{F}(\mathcal{G}_1)$, $(Y_i)_{i\in d}\in \mathcal{F}(\mathcal{G}_2)$ and consider $((X_i)_{i\in d})^{in}$, $((Y_i)_{i\in d})^{in}$. Let $n<|L_{(X_i)_{i\in d}}|=\alpha$ be the length of $((X_i)_{i\in d})^{in}$ and $k<|L_{(Y_i)_{i\in d}}|=\beta$ the length of $((Y_i)_{i\in d})^{in}$. Assume that $k=n$.
     We distinguish the following three cases:
        
     $ \bf{Case\, 1:}$ 
         Let both sets $L_{(X_i)_{i\in d}}\setminus L_{((X_i)_{i\in d})^{in}}$ and $L_{(Y_i)_{i\in d}} \setminus L_{((Y_i)_{i\in d})^{in}}$ be non empty. 
         Pick a finite strong subtree $(Z^1_i)_{i\in d}$ of $(T^0_i)_{i\in d}$ with height $n$. 
          Then by applying Lemma $2$ on the $\alpha-n$, $\beta-n$ uniform families on $L_{(T^0_i[Z^1_i])_{i\in d}}\setminus n$, and the mappings $c''_j, j\in \{1,2\}$, defined by $c''_j(L_j)=c'_j(\mathcal{C}^{(T^0_i)_{i\in d}}_{(Z^1_i)_{i\in d},L_j})$ we get $(T^1_i)_{i\in d}\in \mathcal{S}_{\infty}((T^1_i)_{i\in d})$, where $(T^1_i)_{i\in d}\upharpoonright n=(Z^1_i)_{i\in d}$. $(T^1_i)_{i\in d}$ satisfies the second alternative of Lemma $2$, because we have assumed that $\mathcal{F}(\mathcal{G}_1)\upharpoonright (T^0_i)_{i\in d}\subseteq \mathcal{S}_{1,2}$.
   On $(T^1_i)_{i\in d}$ for $(X_i)_{i\in d}\in \mathcal{F}(\mathcal{G}_1)$ and $(Y_i)_{i\in d}\in \mathcal{F}(\mathcal{G}_2)$ we have that $$\text{If }((X_i)_{i\in d})^{in}=((Y_i)_{i\in d})^{in} =( Z^1_i)_{i\in d}\text{, then }c((X_i)_{i\in d})\neq c((Y_i)_{i\in d}).$$ Set $(T_i)_{i\in d}\upharpoonright n=(T^1_i)_{i\in d}\upharpoonright n=(Z^1_i)_{i\in d}$.
          
          Suppose we have constructed $(T_i)_{i\in d}\upharpoonright (n+m)=(T^m_i)_{i\in d}\upharpoonright (n+m)=(Z^m_i)_{i\in d}$ and we have to decide $(T_i)_{i\in d}\upharpoonright (n+m+1)=(T^{m+1}_i)_{i\in d}\upharpoonright (n+m+1)$.       
          Let now $$(Z^{m}_i)_{i\in d}\sqsubset (Z^{m+1}_i)_{i\in d}\text{ and }(Z^{m+1}_i)_{i\in d}\in\mathcal{S}_{n+m+1}((T^m_i)_{i\in d}).$$ Consider the finite set $ A_{m+1}=\{\, (Z'_i)_{i\in d} \in \mathcal{S}_n((Z^{m+1}_i)_{i\in d})\, \}$. For each $(Z'_i)_{i\in d}\in A_{m+1}$ apply Lemma $2$ on $L_{(T^m_i[Z'_i])_{i\in d}}\setminus n$ and the mappings $c''_j, j\in \{1,2\}$, defined by $c''_j(L_j)=c'_j(\mathcal{C}^{(T^m_i)_{i\in d}}_{(Z'_i)_{i\in d},L_j})$. That gives us $(T'^m_i)_{i\in d}\in \mathcal{S}_{\infty}((T^m_i)_{i\in d})$, where $(T'^m_i)_{i\in d}\upharpoonright (n+m+1)=(Z^{m+1}_i)_{i\in d}$, that satisfies the second alternative of Lemma $2$. On $(T'^{m}_i)_{i\in d}$ for $(X_i)_{i\in d}\in \mathcal{F}(\mathcal{G}_1)$ and $(Y_i)_{i\in d}\in \mathcal{F}(\mathcal{G}_2)$ we have that
          $$\text{If }((X_i)_{i\in d})^{in}=((Y_i)_{i\in d})^{in} =( Z'_i)_{i\in d}\text{, then }c((X_i)_{i\in d})\neq c((Y_i)_{i\in d}).$$ Repeat this step for all the elements of $A_{m+1}$, to get $(T^{m+1}_i)_{i\in d}\in \mathcal{S}_{\infty}((T^m_i)_{i\in d})$ where $T^{m+1}_i\upharpoonright (n+m+1) =Z^{m+1}_i$, for all $i\in d$. On $(T^{m+1}_i)_{i\in d}$ it holds that
          
      $c'_1((X_i)_{i\in d})\neq c'_2((Y_i)_{i\in d})$ for all $(X_i)_{i\in d}\in \mathcal{F}(\mathcal{G}_1)$ and $(Y_i)_{i\in d}\in \mathcal{F}(\mathcal{G}_2)$ with $((X_i)_{i\in d})^{in}=((Y_i)_{i\in d})^{in}=(Z'_i)_{i\in d} \in A_{m+1}$. Set $$(T_i)_{i\in d}\upharpoonright (n+m+1)=(T^{m+1}_i)_{i\in d}\upharpoonright (n+m+1)=(Z^{m+1}_i)_{i\in d}.$$
      
      Let $(T_i)_{i\in d}$ be such that $(T_i)_{i\in d}\upharpoonright n=(T^1_i)_{i\in d}\upharpoonright n=(Z^1_i)_{i\in d}$ and $(T_i)_{i\in d}\upharpoonright n+m=(T^{m}_i)_{i\in d}\upharpoonright n+m$ for all $m\in \omega$. $(T_i)_{i\in d}$ satisfies the conclusions of Lemma $9$, in our case. Suppose not.
        Let $(X_i)_{i\in d}\in \mathcal{F}(\mathcal{G}_1)\upharpoonright (T_i)_{i\in d}$ and $(Y_i)_{i\in d}\in \mathcal{F}(\mathcal{G}_2)\upharpoonright (T_i)_{i\in d}$ with $((X_i)_{i\in d})^{in}=((Y_i)_{i\in d})^{in}=(Z'_i)_{i\in d}\in A_{m'}$ and $c'_1((X_i)_{i\in d})=c'_2((Y_i)_{i\in d})$. But we have that $c'_1((X_i)_{i\in d})\neq c'_2((Y_i)_{i\in d})$ for all $(X_i)_{i\in d}\in \mathcal{F}(\mathcal{G}_1)\upharpoonright (T_i)_{i\in d}$ and $(Y_i)_{i\in d}\in \mathcal{F}(\mathcal{G}_2)\upharpoonright (T_i)_{i\in d}$ with $((X_i)_{i\in d})^{in}=((Y_i)_{i\in d})^{in}=(Z'_i)_{i\in d} \in A_{m'}$, a contradiction.\\

              $\bf{Case\, 2:}$ If now $L_{(Y_i)_{i\in d}}\setminus L_{((Y_i)_{i\in d})^{in}}=\emptyset $.
       Pick a finite strong subtree $(Z^1_i)_{i\in d}$ of $(T^0_i)_{i\in d}$ with height $n$. Let $(X_i)_{i\in d}\in \mathcal{F}(\mathcal{G}_1)$, $(X_i)_{i\in d} \in \mathcal{C}^{(T^0_i)_{i\in d}}_{(N^1_i,L^1_i)_{i\in d}} \subset \mathcal{F}(\mathcal{G}_1)$ with $((X_i)_{i\in d})^{in}=(Z^1_i)_{i\in d}$, and consider the level set $L^X=L_{(X_i)_{i\in d}}\setminus L_{((X_i)_{i\in d})^{in}}\subset L_{(T^0_i[Z^1_i])_{i\in d}}$. If there exists $(Y_i)_{i\in d}$ such that $((Y_i)_{i\in d})^{in}=(Y_i)_{i\in d}=(Z^1_i)_{i\in d}$ and $c'_1((X_i)_{i\in d})=c'_2((Y_i)_{i\in d})$, then subtract a level $l$ from $L_{(T^0_i[Z^1_i])_{i\in d}}$ where $l\in L^{X}$. Let $(T^1_i)_{i\in d}\sqsupseteq (Z^1_i)_{i\in d}$ be a strong subtree of $(T^0_i[Z^1_i])_{i\in d}$ with level set equal to $L_{(T^0_i[Z^1_i])_{i\in d}}\setminus \{l\}$. Then $\mathcal{C}^{(T^1_i)_{i\in d}}_{(N^1_i,L^1_i)_{i\in d}}=\emptyset$. Set $(T_i)_{i\in d}\upharpoonright n= (T^{1}_i)_{i\in d}\upharpoonright n=(Z^1_i)_{i\in d}$. Suppose we have constructed $(T_i)_{i\in d}\upharpoonright m= (T^m_i)_{i\in d}\upharpoonright m= (Z^m_i)_{i\in d}$, $m>n$, and we have to decide $(T_i)_{i\in d}\upharpoonright m+1= (T^{m+1}_i)_{i\in d}\upharpoonright m+1$. Let $(Z^{m+1}_i)_{i\in d}\sqsupset (Z^{m}_i)_{i\in d}$ and $(Z^{m+1}_i)_{i\in d}\in \mathcal{S}_{m+1}((T^m_i)_{i\in d})$. Let $ A=\{\, (Z'_i)_{i\in d} \in \mathcal{S}_n((Z^{m+1}_i)_{i\in d})\, \}$. For each $(Z'_i)_{i\in d}\in A$ if there exists $(Y_i)_{i\in d}\in \mathcal{F}(\mathcal{G}_2)$, $(X_i)_{i\in d}\in \mathcal{F}(\mathcal{G}_1)$ so that $((X_i)_{i\in d})^{in}=(Y_i)_{i\in d}=(Z'_i)_{i\in d}$ and $c'_1((X_i)_{i\in d})=c'_2((Y_i)_{i\in d})$, then subtract a level $l'$ from $L_{(T^m_i[Z'_i])_{i\in d}}$ where $l'\in L^X=L_{(X_i)_{i\in d}}\setminus L_{((X_i)_{i\in d})^{in}}$.   
       Repeat this step for every element of $A$, to get $(T^{m+1}_i)_{i\in d}\in \mathcal{S}_{\infty}((T^m_i)_{i\in d})$ and $T^{m+1}_i\upharpoonright (m+1) =Z^{m+1}_i$ for all $i\in d$. We have that $c'_1((X_i)_{i\in d})\neq c'_2((Y_i)_{i\in d})$ for all $(X_i)_{i\in d}\in  \mathcal{F}(\mathcal{G}_1)\upharpoonright (T^m_i)_{i\in d}$, $(Y_i)_{i\in d}\in  \mathcal{F}(\mathcal{G}_2)\upharpoonright (T^m_i)_{i\in d}$ with $((X_i)_{i\in d})^{in}=((Y_i)_{i\in d}) =(Z'_i)_{i\in d} \in A$. Set $$(T_i)_{i\in d}\upharpoonright (m+1)= (T^{m+1}_i)_{i\in d}\upharpoonright (m+1)= (Z^{m+1}_i)_{i\in d}.$$
       
        Let $(T_i)_{i\in d}$ be such that $(T_i)_{i\in d}\upharpoonright n= (T^{1}_i)_{i\in d}\upharpoonright n=(Z^1_i)_{i\in d}$ and $(T_i)_{i\in d}\upharpoonright m=(T^{m}_i)_{i\in d}\upharpoonright m$ for all $n<m\in \omega$. We claim that it satisfies the conclusions of our lemma in this case. Suppose not.
           Let $(X_i)_{i\in d}\in \mathcal{F}(\mathcal{G}_1)\upharpoonright (T_i)_{i\in d}$,
           $(Y_i)_{i\in d}\in \mathcal{F}(\mathcal{G}_2)\upharpoonright (T_i)_{i\in d}$ with $((X_i)_{i\in d})^{in}=(Y_i)_{i\in d}=(Z''_i)_{i\in d}\in \{\, (Z'_i)_{i\in d} \in \mathcal{S}_n((Z^{m'}_i)_{i\in d})\, \}$ and $c'_1((X_i)_{i\in d})=c'_2((Y_i)_{i\in d})$. By definition we have that $(T_i)_{i\in d}\upharpoonright m'= (T^{m'}_i)_{i\in d}\upharpoonright m'=(Z^{m'}_i)_{i\in d}$. For all $(X_i)_{i\in d}\in  \mathcal{F}(\mathcal{G}_1)\upharpoonright (T_i)_{i\in d}$, $(Y_i)_{i\in d}\in  \mathcal{F}(\mathcal{G}_2)\upharpoonright (T_i)_{i\in d}$ with $((X_i)_{i\in d})^{in}=((Y_i)_{i\in d}) =(Z''_i)_{i\in d}\in \{\, (Z'_i)_{i\in d} \in \mathcal{S}_n((Z^{m'}_i)_{i\in d})\, \}$, it holds that $c'_1((X_i)_{i\in d})\neq c'_2((Y_i)_{i\in d})$, a contradiction.\\

     $\bf{Case\, 3:}$ If $L_{(X_i)_{i\in d}}\setminus L_{((X_i)_{i\in d})^{in}}=\emptyset$ and $L_{(Y_i)_{i\in d}}\setminus L_{((Y_i)_{i\in d})^{in}}=\emptyset$. In the beginning of our lemma we have assumed that $L_{(X_i)_{i\in d}}\neq L_{(Y_i)_{i\in d}}$. The assumption of our case implies that $((X_i)_{i\in d})^{in}=(X_i)_{i\in d}, ((Y_i)_{i\in d})^{in}=(Y_i)_{i\in d}$. We cannot have $((X_i)_{i\in d})^{in}=((Y_i)_{i\in d})^{in}$ because it implies that $(X_i)_{i\in d}=(Y_i)_{i\in d}$ contradicting $\mathcal{F}(\mathcal{G}_1)\upharpoonright (T^0_i)_{i\in d}\subseteq \mathcal{S}_{1,2}$ and the assumption that $L_{(X_i)_{i\in d}}\neq L_{(Y_i)_{i\in d}}$.\\
     
     Suppose now that $\alpha$ and $\beta$ are arbitrary and assume that our lemma holds for any $\gamma$-uniform and $\delta$-uniform families, where $\gamma<\alpha$, $\delta<\beta$. Pick a $t=(t_i)_{i\in d}\in \prod_{i\in d} U_i(n)$. Apply the above assumption on the uniform families $\mathcal{F}(\mathcal{G}_1)(t)$ and $\mathcal{F}(\mathcal{G}_2)(t)$ to get strong subtrees $(T^0_p)_{p\in d\cdot b}$ that satisfy the following property: for $(X^t_j)_{j \in d\cdot b}\in \mathcal{F}(\mathcal{G}_1)(t)\upharpoonright (T^0_p)_{p\in d\cdot b}$ and $(Y^t_j)_{j \in d\cdot b}\in  \mathcal{F}(\mathcal{G}_2)(t)\upharpoonright (T^0_p)_{p\in d\cdot b}$ with $(t^\frown(X^t_j)_{j \in d\cdot b})^{in} =(t^\frown(Y^t_j)_{j \in d\cdot b})^{in}$ we have $c'_1(t^\frown(X^t_j)_{j \in d\cdot b})\neq c'_2(t^\frown(Y^t_j)_{j \in d\cdot b})$.
    Let $$(T^1_i)_{i\in d}=t^\frown (T^0_p)_{p\in d\cdot b}\text{ and }(T_i)_{i\in d}\upharpoonright 1=(T^1_i)_{i\in d}\upharpoonright 1=t.$$

    Suppose we have constructed $(T_i)_{i\in d}\upharpoonright n=(T^n_i)_{i\in d}\upharpoonright n$ and we have to decide $(T_i)_{i\in d}\upharpoonright (n+1)$. Let $\{r_0, \dots , r_{(d\cdot b^{n-1})-1} \}$ be the lexicographically increasing enumeration of the set $\bigcup_{i\in d}T^n(n-1)$. Let $r=(r_{k_i})_{i\in d}$ be so that $r_{k_i}\in T^n_i$ for all $i\in d$. Apply once more our assumption to the uniform families $\mathcal{F}(\mathcal{G}_1)(r)\upharpoonright (T^n_i)_{i\in d}$ and $\mathcal{F}(\mathcal{G}_2)(r)\upharpoonright (T^n_i)_{i\in d}$. After considering all possible such a $r$ we get strong subtrees $(T'^{n+1}_i)_{i\in d\cdot b^n}$. Let $(T^{n+1}_i)_{i\in d}=((T^n_i)_{i\in d}\upharpoonright n)^\frown (T'^{n+1}_i)_{i\in d\cdot b^n}$. Set $(T_i)_{i\in d}\upharpoonright (n+1)=(T^{n+1}_i)_{i\in d}\upharpoonright (n+1)$.
    
     Let $(T_i)_{i\in d}$ be such that $(T_i)_{i\in d}\upharpoonright n=(T^{n}_i)_{i\in d}\upharpoonright n$ for all $n \in \omega$. We claim that it satisfies the conclusion of our lemma. Suppose not.
Let $(X_i)_{i\in d}\in \mathcal{F}(\mathcal{G}_1)\upharpoonright (T_i)_{i\in d}$, $(Y_i)_{i\in d}\in \mathcal{F}(\mathcal{G}_2)\upharpoonright (T_i)_{i\in d}$ with $((X_i)_{i\in d})^{in}=((Y_i)_{i\in d})^{in}=(Z'_i)_{i\in d}$ and $c'_1((X_i)_{i\in d})=c'_2((Y_i)_{i\in d})$. Let $t=(t_0, \dots,t_{d-1})$ be the common root of $(X_i)_{i\in d}$ and $(Y_i)_{i\in d}$. By the definition of $(T_i)_{i\in d}$, for $(X^t_j)_{j \in d\cdot b}\in \mathcal{F}(\mathcal{G}_1)(t)\upharpoonright (T_i)_{i\in d}$ and $(Y^t_j)_{j \in d\cdot b} \in \mathcal{F}(\mathcal{G}_2)(t)\upharpoonright (T_i)_{i\in d}$ with $(t^\frown(X^t_j)_{j \in d\cdot b})^{in} =(t^\frown(Y^t_j)_{j \in d\cdot b})^{in}$ we have that $c'_1(t^\frown(X^t_j)_{j \in d\cdot b})\neq c'_2(t^\frown(Y^t_j)_{j \in d\cdot b})$, a contradiction.

          \end{proof}
 Now we return to the proof of Lemma $8$. 
The idea is to use the above lemma to construct mappings $\lambda_1,\lambda_2$ that have the following property: for every $(X_i)_{i\in d}\in \mathcal{F}( \mathcal{G}_1)\upharpoonright (T^0_i)_{i\in d}$, if there exists $(Y_i)_{i\in d}\in \mathcal{F}( \mathcal{G}_2)\upharpoonright (T^0_i)_{i\in d}$ with $c'_1((X_i)_{i\in d})=c'_2((Y_i)_{i\in d})$ then we would like to pick appropriately a $y\in ((Y_i)_{i\in d})^{in}$ so that $ y\notin \bigcup_i X_i$ and $y$ is a node of any element of $\mathcal{C}^{(T_i)_{i\in d}}_{(N^2_i,L^2_i)_{i\in d}}$. Then set $\lambda_1((X_i)_{i\in d})=y$. By an application of Lemma $7$ we eliminate the possibility of the strong subtree envelope $\mathcal{C}^{(T^0_i)_{i\in d}}_{(N^2_i,L^2_i)_{i\in d}}$ to occur on the resulting infinite strong subtree $(T_i)_{i\in d}$. In other words $\mathcal{C}^{(T_i)_{i\in d}}_{(N^2_i,L^2_i)_{i\in d}}= \emptyset$.
    
   Let $(X_i)_{i\in d} \in \mathcal{F}(\mathcal{G}_1)\upharpoonright (T^0_i)_{i\in d}$ and $(Y_i)_{i\in d} \in\mathcal{F}( \mathcal{G}_2)\upharpoonright (T^0_i)_{i\in d}$ such that $$c'_1((X_i)_{i\in d})=c'_2((Y_i)_{i\in d})$$
   
   If $(X_i)_{i\in d}\in \mathcal{C}^{(T^0_i)_{i\in d}}_{L_1}$ and $(Y_i)_{i\in d}\in \mathcal{C}^{(T^0_i)_{i\in d}}_{L_2}$, where $L_1=\cup_{i\in d}L^1_i$ and $L_2=\cup_{i\in d} L^2_i$, then $L_1\neq L_2$. If $L_1=L_2$ we will have $(X'_i)_{i\in d}\in \mathcal{C}^{(T^0_i)_{i\in d}}_{L_1}$ and $(X'_i)_{i\in d}\in \mathcal{C}^{(T^0_i)_{i\in d}}_{L_2}$ such that $c'_1((X'_i)_{i\in d})=c'_2((X'_i)_{i\in d})$ contradicting that $\mathcal{F}(\mathcal{G}_1)\upharpoonright (T^0_i)_{i\in d}\subseteq \mathcal{S}_{1,2}$.  Assume that $L_1$ is not a proper initial segment of $L_2$, or vice versa.
   If now $L_1\neq L_2$ then $\mathcal{C}^{(T^0_i)_{i\in d}}_{L_1}\neq \mathcal{C}^{(T^0_i)_{i\in d}}_{L_2}$.
   Let $l=\min \{( L_2\setminus L_1)\cup (L_1\setminus L_2)\}$ and assume that $l\in L_2\setminus L_1$. Then for every $(X'_i)_{i\in d}\in \mathcal{C}^{(T^0_i)_{i\in d}}_{L_1}$ , pick $$y\in D=\{\cup_i Y'_i(k): k\in |Y'_i|, (Y'_i)_{i\in d}\in \mathcal{C}^{(T^0_i)_{i\in d}}_{L_2}\}=\cup_{i\in d} T^0_i(l)$$ Set $\lambda_1((X'_i)_{i\in d})=y$. All the members of $D$ are in the image of $\mathcal{C}^{(T^0_i)_{i\in d}}_{L_1}$ under $\lambda_1$. Then by an application of Lemma $7$ we get a strong subtree $(T_i)_{i\in d}$ of $(T^0_i)_{i\in d}$ so that $\mathcal{C}^{(T_i)_{i\in d}}_{L_2}=\emptyset$.
   The possibility of $L_1\sqsubseteq L_2$, or vice versa, is eliminated by the following lemma.
   
   \begin{lemma} By passing to a strong subtree, if necessary, we can assume that on $(T^0_i)_{i\in d}$ there are not two strong subtrees $(X_i)_{i\in d}\in \mathcal{C}^{(T^0_i)_{i\in d}}_{L_1}$, $(Y_i)_{i\in d}\in \mathcal{C}^{(T^0_i)_{i\in d}}_{L_2}$, where $L_1=\bigcup_{i\in d}L^1_i$, $L_2=\bigcup_{i\in d} L^2_i$, such that $L_1\sqsubseteq L_2$ and $c'_1((X_i)_{i\in d})=c'_2((Y_i)_{i\in d})$.
   \end{lemma}
   
   \begin{proof}
We prove the lemma by induction on the countable ordinals $\alpha$, $\beta$ the ranks of $\mathcal{F}(\mathcal{G}_1)$ and $\mathcal{F}(\mathcal{G}_2)$ respectively. If both are finite then for every $t\in \prod_{i\in d} T^0_i(n)$, $n\in \omega$, $\mathcal{T}^t_1$ contains only level sets of a fixed cardinality equal to $\alpha-1$  and $\mathcal{T}^t_2$ contains also level sets of fixed cardinality $\beta-1$. Suppose that $\alpha-1<\beta-1$. Notice that the case of $\alpha=\beta$ is not possible in the above context, since we have assumed that $\mathcal{F}(\mathcal{G}_1)\upharpoonright (T^0_i)_{i\in d}\subseteq \mathcal{S}_{1,2}$. Let $t\in \prod_{i\in d} T^0_i(n)$, for some $n\in \omega$. Pick $(X_i)_{i\in d}\in \mathcal{F}(\mathcal{G}_1)$ with $(X_i(0))_{i\in d}=t$. If there is a $(Y_i)_{i\in d}\in \mathcal{F}(\mathcal{G}_2)$, $(Y_i)_{i\in d}\in \mathcal{C}^{(T^0_i)_{i\in d}}_{(N^2_i,L^2_i)_{i\in d}}$ such that $(X_i)_{i\in d}\sqsubseteq (Y_i)_{i\in d}$ and $c'_1((X_i)_{i\in d})=c'_2((Y_i)_{i\in d})$, then subtract a level $l$ from the level set of $(T^0_i)_{i\in d}$ that is in the level set of $(Y_i)_{i\in d}$ as well, so that $L_{(X_i)_{i\in d}}<l$. Let $(T'^0_i)_{i\in d}$ a resulting strong subtree of $(T^0_i)_{i\in d}$ with 
 $(X_i)_{i\in d}\sqsubseteq  (T'^0_i)_{i\in d}$. Then $\mathcal{C}^{(T'^0_i)_{i\in d}}_{(N^2_i,L^2_i)_{i\in d}}=\emptyset$. Set $$(T_i)_{i\in d}\upharpoonright \alpha= (T'^0_i)_{i\in d}\upharpoonright \alpha=(X_i)_{i\in d}.$$
 
Suppose we have constructed $(T_i)_{i\in d}\upharpoonright (\alpha+n)=(T'^{n}_i)_{i\in d}\upharpoonright (\alpha+n)=(X^{n}_i)_{i\in d}$ and we have to decide $(T_i)_{i\in d}\upharpoonright (\alpha+n+1)$. Let $(X^{n+1}_i)_{i\in d}\sqsupset (X^{n}_i)_{i\in d}$, where $(X^{n+1}_i)_{i\in d}\in \mathcal{S}_{\alpha+n+1}((T'^{n}_i)_{i\in d})$. Let $ B_{n+1}= \{ (X'_i)_{i\in d}\in \mathcal{S}_{\alpha}((X^{n+1}_i)_{i\in d})\}$. For every element $(X'_i)_{i\in d} \in B_{n+1}$, if there exists $(Y_i)_{i\in d}\in \mathcal{F}(\mathcal{G}_2)$, so that $(X'_i)_{i\in d}\sqsubseteq (Y_i)_{i\in d}$ and $c'_1((X'_i)_{i\in d})=c'_2((Y_i)_{i\in d})$ then subtract a level $l$ from the level set of $(T'^n_i)_{i\in d}$ that is in $L_{(Y_i)_{i\in d}}$ as well, so that $L_{(X'_i)_{i\in d}}<l$. Having done that for all elements of $B_{n+1}$ we get $(T'^{n+1}_i)_{i\in d}\in \mathcal{S}_{\infty}((T'^n_i)_{i\in d})$ so that $(X^{n+1}_i)_{i\in d}\sqsubset (T'^{n+1}_i)_{i\in d}$. This strong subtree $(T'^{n+1}_i)_{i\in d}$ has the property that for any element $ (Y_i)_{i\in d}$ of $\mathcal{F}(\mathcal{G}_2)\upharpoonright (T'^{n+1}_i)_{i\in d}$ and $(X'_i)_{i\in d}\in B_{n+1}$, one has that if $c'_1((X'_i)_{i\in d})=c'_2((Y_i)_{i\in d})$ then $L_{(X'_i)_{i\in d}}$ is not an initial segment of $L_{(Y_i)_{i\in d}}$. Set $$(T'_i)_{i\in d}\upharpoonright (\alpha+n+1)=(T'^{n+1}_i)_{i\in d}\upharpoonright (\alpha+n+1)=(X^{n+1}_i)_{i\in d}.$$

 Let $(T_i)_{i\in d}$ be such that $(T_i)_{i\in d}\upharpoonright \alpha= (T'^0_i)_{i\in d}\upharpoonright \alpha$ and $(T_i)_{i\in d}\upharpoonright (\alpha +n)=(T'^{n}_i)_{i\in d}\upharpoonright (\alpha +n)$ for all $n\in \omega$. We claim that it satisfies the conclusions of our lemma. Suppose not. Let $(X_i)_{i\in d}\in \mathcal{F}(\mathcal{G}_1)\upharpoonright (T_i)_{i\in d}$, $(Y_i)_{i\in d}\in \mathcal{F}(\mathcal{G}_2)\upharpoonright (T_i)_{i\in d}$ with $(X_i)_{i\in d}\sqsubseteq (Y_i)_{i\in d}$ and $c'_1((X_i)_{i\in d})=c'_2((Y_i)_{i\in d})$. Then $(X_i)_{i\in d}\in  B_{n'}$ for some $n'\in \omega$. This implies that $L_{(X_i)_{i\in d}}$ is not an initial segment of $L_{(Y_i)_{i\in d}}$, a contradiction.

 Consider arbitrary countable ordinals $\alpha$ and $\beta$ and assume that our lemma holds for every $\delta < \alpha$ and $\gamma<\beta$ uniform families. Pick once more $t=(t_0,\dots, t_{d-1})\in \prod_{i\in d} T^0_i(n)$. By definition $\mathcal{F}(\mathcal{G}_1)(t)$ and $\mathcal{F}(\mathcal{G}_2)(t)$ are of ranks $\delta$ and $\gamma$ so the inductive hypothesis gives us $(T^1_i)_{i\in d\cdot b}$ strong subtree of $(T^0_i)_{i\in d}(t)$ that satisfies the following property: For $(X_i)_{i\in d\cdot b}\in \mathcal{F}(\mathcal{G}_1)(t)\upharpoonright (T^1_i)_{i\in d\cdot b}$ and $(Y_i)_{i\in d\cdot b}\in  \mathcal{F}(\mathcal{G}_2)(t)\upharpoonright (T^1_i)_{i\in d\cdot b}$ if we have $c'_1(t^\frown (X_i)_{i\in d\cdot b})=c'_2(t^\frown (Y_i)_{i\in d\cdot b})$ then $L_{(X_i)_{i\in d\cdot b}}$ is not an initial segment of $L_{(Y_i)_{i\in d\cdot b}}$. Set $(T'^2_i)_{i\in d}=t^\frown (T^1_i)_{i\in d\cdot b}$ and $(T_i)_{i\in d}\upharpoonright 2=(T'^2_i)_{i\in d}\upharpoonright 2$. Suppose we have constructed $(T_i)_{i\in d}\upharpoonright n= (T'^{n}_i)_{i\in d}\upharpoonright n$ and we have to decide $(T_i)_{i\in d}\upharpoonright (n+1)$. 
 
 Consider the set $ H=\{ \bigcup_{i\in d} T'^{n}_i(n-1)\}$. For any $r=(r_0,\dots, r_{d-1})\subset H$, where $r_i\in T'^{n}_i(n-1)$ for all $i\in d$, $\mathcal{F}(\mathcal{G}_1)(r)\upharpoonright (T'^n_i)_{i\in d}(r)$ and $\mathcal{F}(\mathcal{G}_2)(r)\upharpoonright (T'^n_i)_{i\in d}(r)$ are of ranks $\delta$ and $\gamma$. The inductive hypothesis gives us strong subtrees $(T'^{r}_i)_{i\in d\cdot b}\in \mathcal{S}_{\infty}( (T'^n_i)_{i\in d}(r))$ that satisfy the following: For any $(Z_i)_{i\in d\cdot b}\in \mathcal{F}(\mathcal{G}_1)(r)\upharpoonright (T'^{r}_i)_{i\in d\cdot b}$ and $(Y_i)_{i\in d\cdot b}\in \mathcal{F}(\mathcal{G}_2)(r)\upharpoonright (T'^{r}_i)_{i\in d\cdot b}$, $$\text{if }c'_1(r^\frown (Z_i)_{i\in d\cdot b})=c'_2(r^\frown (Y_i)_{i\in d\cdot b})\text{, then }L_1\text{ is not an initial segment of } L_2$$ for $L_1$ being the level set of $(Z_i)_{i\in d\cdot b}$ and $L_2$  the one of $(Y_i)_{i\in d\cdot b}$. Repeat the above step for any such an $r$ to get strong subtrees $(T''_i)_{i\in d\cdot b^n}$. Set $$(T'^{n+1}_i)_{i\in d}=((T'^{n}_i)_{i\in d}\upharpoonright n)^\frown (T''_i)_{i\in d\cdot b^n}\text{ and }(T_i)_{i\in d}\upharpoonright (n+1)=(T'^{n+1}_i)_{i\in d}\upharpoonright (n+1).$$
 
 Let $(T_i)_{i\in d}\upharpoonright n= (T'^{n}_i)_{i\in d}\upharpoonright n$, for all $n\in \omega$. $(T_i)_{i\in d}\in \mathcal{S}_{\infty}((T^0_i)_{i\in d})$ satisfies the conclusions of our lemma. Suppose not. Let $(X_i)_{i\in d}\in \mathcal{F}(\mathcal{G}_1)\upharpoonright (T_i)_{i\in d}$, $(Y_i)_{i\in d}\in \mathcal{F}(\mathcal{G}_2)\upharpoonright (T_i)_{i\in d}$ with $(X_i)_{i\in d}\sqsubseteq (Y_i)_{i\in d}$ and $c'_1((X_i)_{i\in d})=c'_2((Y_i)_{i\in d})$. Let $t=(X_i(0))_{i\in d}=(Y_i(0))_{i\in d}$. On $ (T_i)_{i\in d}(t)$ we have that if $c'_1((X_i)_{i\in d}=t^\frown(X'_i)_{i\in d\cdot b})=c'_2(t^\frown(Y'_i)_{i\in d\cdot b}=(Y_i)_{i\in d})$, then $L_{(X'_i)_{i\in d\cdot b}}$ is not an initial segment of $L_{(Y'_i)_{i\in d\cdot b}}$, a contradiction.

   \end{proof}

   Now we return to the proof of Lemma $8$. Let $(X_i)_{i\in d} \in \mathcal{F}(\mathcal{G}_1)\upharpoonright (T^0_i)_{i\in d}$ and $(Y_i)_{i\in d} \in\mathcal{F}( \mathcal{G}_2)\upharpoonright (T^0_i)_{i\in d}$, with $(X_i)_{i\in d}\in \mathcal{C}^{(T^0_i)_{i\in d}}_{(N^1_i,L^1_i)_{i\in d}}$ and $(Y_i)_{i\in d}\in \mathcal{C}^{(T^0_i)_{i\in d}}_{(N^2_i,L^2_i)_{i\in d}}$. Let also $$c'_1((X_i)_{i\in d})=c'_2((Y_i)_{i\in d})$$
   If both sets $\bigcup_{i\in d}N^1_i$ and $\bigcup_{i\in d} N^2_i$ are nonempty 
   there are the following possibilities. Firstly $$\cup_{i\in d}N^1_i\neq \cup_{i\in d} N^2_i\text{ and } L^1_{in}\neq \emptyset \text{ or }L^2_{in}\neq \emptyset$$ If there exists either $y\in \bigcup_{i\in d} N^2_i$ so that $y\notin ((X'_i)_{i\in d})$, for a $(X'_i)_{i\in d}\in \mathcal{C}^{(T^0_i)_{i\in d}}_{( N_i^1, L^1_i)_i}$, or $x \in \bigcup_{i\in d} N^1_i$ so that $x\notin ((Y'_i)_{i\in d})$, for a $(Y'_i)_{i\in d}\in \mathcal{C}^{(T^0_i)_{i\in d}}_{( N^2_i,  L^2_i)_i}$, then set $$\lambda_1((X'_i)_{i\in d})=y\text{ or }\lambda_2((Y'_i)_{i\in d})=x$$ respectively. If no such an $y$ or $x$ are possible to be found and since by Lemma $9$ we have that $((X_i)_{i\in d})^{in}\neq ((Y_i)_{i\in d})^{in}$, we conclude that $L_{in}^1\neq L_{in}^2$.
    
  Suppose that $L_{in}^1\neq L_{in}^2$. 
  In this case let $$l=\min \{ (L^1_{in}\setminus L^2_{in})\cup (L^2_{in}\setminus L^1_{in})\}$$ Suppose that $l\in L^1_{in}$. Identical argument holds if $l\in L^2_{in}$. Consider the set $$D=\{ x\in \cup_{i\in d} T^0_i(l): (\exists x'\in \cup_{i\in d} N^1_i) x\leq x'\}$$ Then for every $(Y'_i)_{i\in d}\in \mathcal{C}^{(T^0_i)_{i\in d}}_{(N^2_i,L^2_i)_{i\in d}}$ pick an $x\in D$ and set $\lambda_2((Y'_i)_{i\in d})=x$. Notice that every element of $D$ is a node of any strong subtree of the strong subtree envelope $\mathcal{C}^{(T^0_i)_{i\in d}}_{(N^2_i,L^2_i)_{i\in d}}$. Identical argument applies in the case of $\bigcup_{i\in d}N^1_i= \bigcup_{i\in d} N^2_i$. In this case by Lemma $9$ we must have $L_{in}^1\neq L_{in}^2$.
  
  Consider the case that $\bigcup_{i\in d}N^1_i\neq \bigcup_{i\in d} N^2_i$ and $L^1_{in}=L^2_{in}=\emptyset$. $((X_i)_{i\in d})^{in}\neq ((Y_i)_{i\in d})^{in}$ implies that there exists either $y\in \bigcup_{i\in d} N^2_i$ so that $y\notin (X'_i)_{i\in d}$, or $x \in \cup_{i\in d} N^1_i$ so that $x\notin (Y'_i)_{i\in d}$, for $(X'_i)_{i\in d}$ a member of $\mathcal{C}^{(T^0_i)_{i\in d}}_{(N_i,L_i)_{i\in d}}$ and $(Y'_i)_{i\in d}$ a member of $\mathcal{C}^{(T^0_i)_{i\in d}}_{(N_i,L_i)_{i\in d}}$. We set $\lambda_1((X'_i)_{i\in d})=y$ or $\lambda_2((Y'_i)_{i\in d})=x$.
  If not such an $x$ or $y$ is possible to be fund, then $\bigcup_{i\in d}N^1_i$ is in any strong subtree of $\mathcal{C}^{(T^0_i)_{i\in d}}_{(N^2_i,L^2_i)_i}$ and $\bigcup_{i\in d} N^2_i$ is in any strong subtree of $\mathcal{C}^{(T^0_i)_{i\in d}}_{(N^1_i,L^1_i)_i}$. This implies that $((X_i)_{i\in d})^{in}=((Y_i)_{i\in d})^{in}$, a contradiction with Lemma $9$. 
    
  Lastly if $\bigcup_{i\in d}N^1_i\neq \emptyset$ and $\bigcup_{i\in d} N^2_i=\emptyset$. In the case that $((Y_i)_{i\in d})^{in}$ is not defined, there will be an $x\in \bigcup_{i\in d}N^1_i$ so that $x\notin (Y'_i)_{i\in d}$ for some $(Y'_i)_{i\in d}\in \mathcal{C}^{(T^0_i)_{i\in d}}_{\cup_{i\in d}L^2_i}$. To see this notice that the strong subtree envelope $ \mathcal{C}^{(T^0_i)_{i\in d}}_{\cup_{i\in d}L^2_i}$ is taken over the level set $\cup_{i\in d}L^2_i$. As a result we can choose a $(Y'_i)_{i\in d}\in \mathcal{C}^{(T^0_i)_{i\in d}}_{\cup_{i\in d}L^2_i}$ so that $x\notin \cup_{i\in d}Y'_i$.
  Set $\lambda_2((Y'_i)_{i\in d})=x$. If now $((Y_i)_{i\in d})^{in}$ is defined, since $((X_i)_{i\in d})^{in}\neq ((Y_i)_{i\in d})^{in}$, if we cannot choose such an $x$, then we will be able to choose $y\in ((Y_i)_{i\in d})^{in}$ and set $\lambda_1((X_i)_{i\in d})=y$.

    The above show that we can construct mappings $\lambda_1,\lambda_2$ such that by two consecutive applications of Lemma $7$ we get $(T_i)_{i\in d}\in S_{\infty}((U_i)_{i\in d})$ that $$\lambda_j(\mathcal{F}(\mathcal{G}_j)\upharpoonright (T_i)_{i\in d}))\cap (T_i)_{i\in d}=\emptyset \text{, }j\in \{1,2\}$$ Suppose that $c_1((X_i)_{i\in d})=c_2((Y_i)_{i\in d})$ for some $(X_i)_{i\in d}\in \mathcal{F}(\mathcal{G}_1)\upharpoonright (T_i)_{i\in d}$ and a $(Y_i)_{i\in d}\in \mathcal{F}(\mathcal{G}_2)\upharpoonright (T_i)_{i\in d}$, where $(X_i)_{i\in d}\in \mathcal{C}^{(T_i)_{i\in d}}_{(N^1_i,L^1_i)_{i\in d}}$ and $(Y_i)_{i\in d}\in \mathcal{C}^{(T_i)_{i\in d}}_{(N^2_i,L^2_i)_{i\in d}}$.
  
    This contradicts the way that $\lambda_1,\lambda_2$ are defined. We must have either $\mathcal{C}^{(T_i)_{i\in d}}_{(N^1_i,L^1_i)_i}=\emptyset$ or $\mathcal{C}^{(T_i)_{i\in d}}_{(N^2_i,L^2_i)_i}=\emptyset$. Therefore $(T_i)_{i\in d}$ satisfies the second alternative of our lemma. 
  \end{proof}
    
     We make the following observation:
     
     \begin{lemma} Under the assumptions of Lemma $8$, if $\mathcal{F}(\mathcal{G}_1)$ is an $\alpha$-uniform family, $\mathcal{F}(\mathcal{G}_2)$ is a $\beta$-uniform, with $\alpha \neq \beta$,then the first statement of the lemma is excluded. \end{lemma}
     
     \begin{proof} It is an easy inductive argument that if $\mathcal{G}$ is an $\alpha$-uniform cannot be $\beta$-uniform, for any $\beta\neq \alpha$. \end{proof}

   Finally we are able do the inductive step of Theorem $7$ for any $\alpha$-uniform family on $U$.\\

 \subsection{ Inductive step}

    Let  $\mathcal{G}$ be an $\alpha$-uniform family of finite strong subtrees of $U$. For any $t\in U$, $\mathcal{G}(t)$ is a $\beta$-uniform family on $U(t)$ for some $\beta<\alpha$. Therefore by the inductive hypothesis we can assume that the coloring $c_t$ defined on $\mathcal{G}(t)$ by $c_t((X_i)_{i\in b})=c(t^\frown (X_i)_{i\in b})$, is canonical. As a consequence at each node $t$ of $U$ we have a uniform family $\mathcal{F}(\mathcal{G})(t)$, that results by taking the union of the strong subtree envelopes of all members of $\mathcal{T}^t$, together with $f_t$ and a one-to-one mapping $\phi_t$ that witness the coloring $c_t$ being canonical on $U(t)$. As we have mentioned above $c_t$ is defined on $\mathcal{G}(t)$ by $c_t((X_i)_{i\in b})=\phi_t(f_t((X_i)_{i\in b})=(N_i,L_i)_{i\in b})$ where $(X_i)_{i\in b}\in \mathcal{G}(t)$, $(N_i,L_i)_{i\in b} \in \mathcal{T}^t$ and $\mathcal{C}^{U(t)}_{(N_i,L_i)_{i\in d}}\subset \mathcal{F}(\mathcal{G})(t)$.\\

     We will construct the strong subtree $T$ that satisfies the conclusion of the Theorem $7$, by applying continuously Lemma $8$. Pick a node $r\in U$ and set $T(0)=r$. Let $(r^\frown i)_{i\in b}$ be the set of the immediate successors of $r$ in $U$ and let $T^2=U[r]$. Set $T(1)=(r^\frown i)_{i\in b}$. Equivalently $T\upharpoonright 2=T^2\upharpoonright 2$. Suppose we have constructed $T \upharpoonright n=T^n\upharpoonright n$ and we have to decide $T \upharpoonright (n+1)$. 
     
     Let $ T^n(n-1)=(r_p)_{p \in b^{n-1}}$. Consider the uniform families $\mathcal{F}(\mathcal{G})(r_p)$ on $T^n(r_p)$, for all $p\in b^{n-1}$.      
          For any pair $\mathcal{F}(\mathcal{G})(r_i)$, on $T^n(r_i)$ and $\mathcal{F}(\mathcal{G})(r_j)$ on $T^n(r_j)$, $i,j\in b^{n-1}$, apply Lemma $8$ up to translation. Having done that for all possible such pairs, we get strong subtrees $(T'^1_m)_{m\in b^n}\in \mathcal{S}_{\infty}((T^n(r_p))_{p\in b^{n-1}})$ that satisfy either the first or the second alternative of Lemma $8$.
     Consider the uniform families $\mathcal{F}(\mathcal{G})(r_0)\upharpoonright (T'^1_m)_{m\in b^n}$ and $\mathcal{F}(\mathcal{G})(s)\upharpoonright (T'^1_m)_{m\in b^n}$, for $s\in T^n(n')$, $n'<n-1$. There exists a $k=b^{n-1-n'}$ and $l\in \omega$, so that $\{( r_p)_{p\in [l\cdot b,  (l\cdot b)+k)}\}=T^n(n-1)\cap T^n(s)$. In other words $s^\frown i$ has $k/b$ many successors on $T^n(n-1)$. These successors are precisely: $(r_p)_{p\in [( l\cdot b)+(i\cdot k/b), ( l\cdot b)+(i\cdot k/b)+k/b)}$. 
     As a result $(T'^1_m)_{m \in [l\cdot b^2,  (l\cdot b^2)+k\cdot b)}\in \mathcal{S}_{\infty}(T^n(s))$.

     Apply Lemma $8$ on the uniform family $\mathcal{F}(\mathcal{G})(r_0)\upharpoonright (T'^1_m)_{m\in [0, b)}$ and the family $\pi_m(\mathcal{F}(\mathcal{G})(s)\upharpoonright T'^1_{(l\cdot b^2)+(m\cdot k)})$ translated on $T'^1_m$, for all $m \in b$. If we have the first alternative of Lemma $8$ holding, we proceed to the node $r_1$. Otherwise we consider the uniform families $\mathcal{F}(\mathcal{G})(r_0)\upharpoonright( T'^1_m)_{m\in [0,b)}$ and $\pi_m(\mathcal{F}(\mathcal{G})(s)\upharpoonright (T'^1_{(l\cdot b^2+1)+(m \cdot k)})$ translated on $T'^1_m$, for all $m\in b$. Once again if we get the first statement of Lemma $8$, we proceed to the node $r_1$, otherwise we apply again Lemma $8$ to the uniform families $\mathcal{F}(\mathcal{G})(r_0)\upharpoonright (T'^1_m)_{m\in b}$ and $\pi_m(\mathcal{F}(\mathcal{G})(s)\upharpoonright T'^1_{(l\cdot b^2+2)+(m \cdot k)}$ translated on $T'^1_m$, for all $m \in b$, etc. Having done that for the finite set of all possible pairs of nodes $r_p$ and $s$, we get strong subtrees $(T'^n_m)_{m\in b^{n}}\in \mathcal{S}_{\infty}((T'^1_m)_{m\in b^n})$ such that for any two uniform families $\mathcal{F}(\mathcal{G})(r_{p})$ and $\mathcal{F}(\mathcal{G})(s)$ we have either the first or the second statement of Lemma $8$ holding.
     
     Suppose that we get always the first statement of Lemma $8$. In this case let $T^{n+1}=(T^n\upharpoonright n)^\frown (T'^n_m)_{m\in b^{n}}$. Set $T\upharpoonright n+1=T^{n+1}\upharpoonright n+1$.
     
      If the second statement of Lemma $8$ occurs, we distinguish two cases: first if it occurs on an application of Lemma $8$ on $\mathcal{F}(\mathcal{G})(r_i)$ and $\mathcal{F}(\mathcal{G})(r_{j})$, $i,j\in b^{n-1}$. This case has no impact on the argument, since $c(\mathcal{F}(\mathcal{G})(r_{i})\upharpoonright \tilde{T}^n)\cap c(\mathcal{F}(\mathcal{G})(r_{j}\upharpoonright \tilde{T}^n)=\emptyset$, where $\tilde{T}_n=(T^n\upharpoonright n)^\frown (T'^n_m)_{m\in b^{n}}$.

  Secondly if it occurs on an application of Lemma $8$ on the uniform families $\mathcal{F}(\mathcal{G})(r_{p})$ and $\mathcal{F}(\mathcal{G})(s)$.  In this case we have to reassure that if $c(\mathcal{F}(\mathcal{G})(r_{p})\upharpoonright \tilde{T}^n)\cap c(\mathcal{F}(\mathcal{G})(s)\upharpoonright (T'^1_m)_{m \in [l\cdot b^2,  (l\cdot b^2)+k\cdot b)})=\emptyset$, then $c(\mathcal{F}(\mathcal{G})(r_{p}))\cap c(\mathcal{F}(\mathcal{G})(s))=\emptyset$ on an infinite strong subtree of $\tilde{T}^n$.
     
      At first notice that there are at most finitely many strong subtrees $X_s=(X'_i)_{i\in b}$ members of $ \mathcal{F}(\mathcal{G})(s)\upharpoonright (T'^1_m)_{m \in [l\cdot b^2,  (l\cdot b^2)+k\cdot b)}$ with $L_{X_s}< n$. We can eliminate the possibility of any strong subtree $X_s=(X'_i)_{i\in b}$, with $L_{X_s}<n$, that corresponds to the uniform family $\mathcal{F}(\mathcal{G})(s)$, having the same color with a strong subtree $X_{r_p}=(Y'_i)_{i\in b}\in \mathcal{F}(\mathcal{G})(r_p)$. We do that by simply eliminating a level $l$ from the level set $L_{ \tilde{T}^n[r_{p}]}$ so that $l\in L_{(N^{r_{p}}_i, L^{r_{p}}_i)_i}\cap L_{ \tilde{T}^n[r_{p}]}$ where $(Y'_i)_{i\in b}\in \mathcal{C}^{\tilde{T}^n}_{(N^{r_{p}}_i, L^{r_{p}}_i)_i}$. In any of the resulting strong subtrees $T'$ of $\tilde{T}^n[r_{p}]$, with $L_{T'}=L_{\tilde{T}^n[r_{p}]}\setminus \{l\}$, we have that $\mathcal{C}^{T'}_{(N^{r_{p}}_i, L^{r_{p}}_i)_i}=\emptyset$.
      For notational simplicity we are going to use $X_s, X_{r_{p}}$ instead of $(X'_i)_{i\in b}$ and $(Y'_i)_{i\in b}$ respectively.\\
           
     There may be a strong subtree $X_s$ with a level set that contains both levels smaller than $n$ and bigger as well. In that case we restrict on $Y$ the initial segment of $X_s$ with level set that lies below $n$ i.e. $Y\sqsubset X$ and $L_Y<n$. Observe that $\mathcal{F}(\mathcal{G})(s)(Y)$ contains $d>b$ sequences of finite strong subtrees. Notice that $d$ is a multiple of $b$. In that case we need an extended version of Lemma $8$ as follows:
     
     \begin{lemma}  Let $(U_i)_{i\in d}$, where $d=k b$ is a multiple of $b$, the branching number of $U_i$ for all $i$.  Let $\mathcal{T}_1$ be a family of node-level sets on $(U_i)_{i\in d'}$ where $d'\subset d$, that generates an $\beta$-uniform family $ \mathcal{F}(\mathcal{G}_1)$ on $(U_i)_{i\in d'}$. Let $\mathcal{T}_2$ be a family of node-level sets on $(U_i)_{i\in d}$, that generates an $\alpha$-uniform family $\mathcal{F}(\mathcal{G}_2)$ on $(U_i)_{i\in d}$, for $\alpha>\beta$. Let $c'_1$ a mapping on $\mathcal{F}(\mathcal{G}_1)$ with the property that $c'_1((X^1_i)_{i\in d'})=c'_1((X^2_i)_{i\in d'})$ if and only if $(X^1_i)_{i\in d'}:(N^1_i, L^1_i)_{i\in d'}=(X^2_i)_{i\in d'}:(N^1_i, L^1_i)_{i\in d'}$ for $(N^1_i, L^1_i)_{i\in d'}\in \mathcal{T}_1$. Let also $c_2$ a mapping on $\mathcal{F}(\mathcal{G}_2)$ such that $c_2((Y^1_i)_{i\in d})=c_2((Y^2_i)_{i\in d})$ if and only if $(Y^1_i)_{i\in d}:(N^2_i, L^2_i)_{i\in d}=(Y^2_i)_{i\in d}:(N^2_i, L^2_i)_{i\in d}$ for $(N^2_i, L^2_i)_{i\in d}\in \mathcal{T}_2$.
     
 There exists a strong subtree $(T_i)_{i\in d}$ of $(U_i)_{i\in d}$ such that the following holds:
     $$c'_1(\mathcal{F}( \mathcal{G}_1)\upharpoonright (T_i)_{i\in d'}) \cap c'_2(\mathcal{F}(\mathcal{G}_2)\upharpoonright (T_i)_{i\in d})=\emptyset.$$
   \end{lemma}
   \begin{proof} Notice that we can extend $(\mathcal{F}(\mathcal{G}_1),c'_1)$ on $(U_i)_{i\in d}$ by 
   $c'_1((X_i)_{i\in d})=c'_1((X_j)_{j\in d'})$ and $X_i=X_j$ for $j\in d'$. Then apply Lemma $8$ and Lemma $11$.
   
   \end{proof}
     
      We can consider now the corresponding uniform families $\mathcal{F}(\mathcal{G})(s)(Y)$ on $\tilde{T}^n(Y)$ and $\mathcal{F}(\mathcal{G})(r_p)$ on $\tilde{T}^n(r_p)$. Then apply Lemma $12$ to get a strong subtree that satisfies its conclusion. Repeating that for the finite set of all $X_s\in \mathcal{F} (\mathcal{G})(s)\upharpoonright \tilde{T^n}$ whose set of levels intersects $[n,\infty)$, we succeed in getting a strong subtree $T^{n+1}$ of $\tilde{T}^n$ such that $$c(\mathcal{F}(\mathcal{G})(s)\upharpoonright T^{n+1} )\cap c(\mathcal{F}(\mathcal{G})(r_p)\upharpoonright T^{n+1})=\emptyset$$
      
      Set $T\upharpoonright n+1=T^{n+1}\upharpoonright n+1$.
       
           Proceeding in that manner we construct $T\in \mathcal{S}_{\infty}(U)$, where $T \upharpoonright n=T^{n}\upharpoonright n$, for all $n\in \omega$,
          such that for any two nodes $s_0,s_1\in T$, with $|s_0|\leq |s_1|$, we have one of the two following alternatives.
           \begin{enumerate}
           \item{} There exists $(T^{s_0}_i)_{i\in b}\in \mathcal{S}_{\infty}(T(s_0))$ such that $\mathcal{F}(\mathcal{G})(s_0)\upharpoonright (T^{s_0}_i)_{i\in b}=\mathcal{F}(\mathcal{G})(s_1)$, up to translation. Also for every $X\in \mathcal{F}(\mathcal{G})(s_0)\upharpoonright (T^{s_0}_i)_{i\in b}, Y\in \mathcal{F}(\mathcal{G})(s_1)$, with $Y$ a translate of $X$, it holds that $c(X)=c(Y)$.
           \item{} $c(\mathcal{F}(\mathcal{G})(s_0))\cap c(\mathcal{F}(\mathcal{G})(s_1))=\emptyset$.
           \end{enumerate}
           
           To define precisely the family of node-level sets $\mathcal{T}$ that will satisfy the conclusions of Theorem $7$ we need the following result.
              \begin{proposition} Let $\mathcal{T}_1$ and $\mathcal{T}_2$ be two families of node-level sets that generate two uniform families $\mathcal{F}(\mathcal{G}_1)$ and $\mathcal{F}(\mathcal{G}_2)$ on $U$ by taking the union of all strong subtree envelopes of all node-level sets of $\mathcal{T}_1$ and $\mathcal{T}_2$ respectively. Let $c_1$ a mapping on $\mathcal{F}(\mathcal{G}_1)$ with the property that $c_1(X_1)=c_1(X_2)$ if and only if $X_1:(N_1, L_1)=X_2:(N_1, L_1)$ for $(N_1, L_1)\in \mathcal{T}_1$. Let also $c_2$ a mapping on $\mathcal{F}(\mathcal{G}_2)$ so that $c_2(Y_1)=c_2(Y_2)$ if and only if $Y_1:(N_2, L_2)=Y_2:(N_2, L_2)$ for $(N_2, L_2)\in \mathcal{T}_2$.
                     If by an application of Lemma $8$ we get a $T\in \mathcal{S}_{\infty}(U)$ such that the first alternative holds, then we have that $\mathcal{T}_1\upharpoonright T=\mathcal{T}_2\upharpoonright T$. 
       \end{proposition}
       
       \begin{proof} The proof is by induction on the rank $\alpha$ of the uniform families $\mathcal{F}(\mathcal{G}_1)$ and $\mathcal{F}(\mathcal{G}_2)$, which is identical by Lemma $11$. If $\alpha\in \omega$, then by the discussion before Lemma $6$ we have that for any $(N_0,L_0),(N_1,L_1)\in \mathcal{T}_1$ and any two members of their strong subtree envelopes $X_0\in \mathcal{C}^{T}_{(N_0, L_0)}$ and $X_1\in \mathcal{C}^{T}_{(N_1, L_1)}$ one has: $\iota_{b^{\alpha},X_0} \circ \iota^{-1}_{b^{\alpha},X_1}(N_1)=(N_0)$ and $| L_0|= | L_1 |$. Similarly for $\mathcal{T}_2$. Suppose that $\mathcal{T}_1\upharpoonright T\neq \mathcal{T}_2\upharpoonright T$. Let $(N_1,L_1)\in \mathcal{T}_1\upharpoonright T$, so that if $|N_1|>1$ then for every $t, t' \in N_1$, the absolute value of the difference $|t|-|t'|$ is greater than $1$. For any $X\in \mathcal{C}^{T}_{(N_1, L_1)}$ consider $c_1(X)$. Since we have the first alternative of Lemma $8$ on hold, we must have that $c_2(X)=c_1(X)$ for $X \in \mathcal{C}^{T}_{(N_2, L_2)}$, $(N_2,L_2)\in \mathcal{T}_2\upharpoonright T$ as well. That must be true for all the members of $ \mathcal{C}^{T}_{(N_1, L_1)}$, which implies that $ \mathcal{C}^{T}_{(N_1, L_1)}= \mathcal{C}^{T}_{(N_2, L_2)}$. As a result for $X\in \mathcal{C}^{T}_{(N_1, L_1)}$ and $Y\in \mathcal{C}^{T}_{(N_2, L_2)}$ we have that $L_X=L_Y$.
       If $N_1=\{ t\}$, then $N_2=\{ t\}$ as well, otherwise if $N_2=\{ s\}$, then $X(0)=t\neq s=X(0)$, a contradiction. Suppose that $| N_1|>1$ and let $t\in \{ (N_1\setminus N_2)\cup (N_2\setminus N_1)\}$ is of minimal height. Suppose that $t\in N_1$. Since $X\in \mathcal{C}^{T}_{(N_1, L_1)}$ there exists $n\in |X|$ such that $t\in X(n)$. Choose a $Y\in  \mathcal{C}^{T}_{(N_2, L_2)}$ such that $t\notin Y(n)$. Notice that $Y\notin \mathcal{C}^{T}_{(N_1, L_1)}$, a contradiction. If now $N_1=N_2=\emptyset$ then we must have $L_1=L_2$, other wise for every $X'\in  \mathcal{C}^{T}_{(N_1, L_1)}$ and $Y'\in \mathcal{C}^{T}_{(N_2, L_2)}$ we would have that $L_X'\neq L_Y'$ contradicting that $ \mathcal{C}^{T}_{(N_1, L_1)}= \mathcal{C}^{T}_{(N_2, L_2)}$. Finally if $N_2=\emptyset$ and $N_1\neq \emptyset$ pick $t\in N_1$ so that for any other $t' \in N_1$, we have that $l=|t|\geq |t'|$. Pick a $Y\in \mathcal{C}^{T}_{ L_2}$ so that $t\notin Y$. This is always possible since our node-level set $(N_2,L_2)$ is only a level set.
       Then $Y\notin \mathcal{C}^{T}_{(N_1, L_1)}$, a contradiction of $ \mathcal{C}^{T}_{(N_1, L_1)}= \mathcal{C}^{T}_{(N_2, L_2)}$. As a consequence $ \mathcal{C}^{T}_{(N_1, L_1)}= \mathcal{C}^{T}_{(N_2, L_2)}$ implies that  for any $X'\in \mathcal{C}^{T}_{(N_1, L_1)}$, $Y'\in\mathcal{C}^{T}_{(N_2, L_2)}$ both finite strong subtrees of height $\alpha<\omega$, we have that $\iota_{b^{\alpha},X'} \circ \iota^{-1}_{b^{\alpha},Y'}(N_2)=(N_1)$ and $| L_1|= | L_2 |$. Therefore $\mathcal{T}_1\upharpoonright T=\mathcal{T}_2\upharpoonright T$.

            Assume that the assertion of our proposition holds for $\beta< \alpha$ uniform families and consider the case of $\alpha\geq \omega$ uniform families $\mathcal{F}(\mathcal{G}_1)$ and $\mathcal{F}(\mathcal{G}_2)$. For any node $t\in T$, $\mathcal{F}(\mathcal{G}_1)(t)\upharpoonright T$ and $\mathcal{F}(\mathcal{G}_2)(t)\upharpoonright T$ are both uniform families of rank less than $\alpha$. The inductive hypothesis applies to give us $\mathcal{T}^t_1\upharpoonright T=\mathcal{T}^t_2\upharpoonright T$. That being true for every $t\in T$ implies that $\mathcal{T}_1\upharpoonright T=\mathcal{T}_2\upharpoonright T$.
       \end{proof}

           Now the family of node-level sets $\mathcal{T}$ that will satisfy the conditions of Definition $16$ is defined as follows:  For a node $s_0\in T$ if there exists a node $s_1\in T$ so that the first alternative of the above statement holds, then $\mathcal{T}^{s_0}\subset \mathcal{T}$. If for all $s_1\in T$ we have the second alternative holding then $s_0\cup \mathcal{T}^{s_0}:=\{ (s_0\cup N,L): (N,L)\in \mathcal{T}_{s_0}\}\subset \mathcal{T}$. Similarly for $\phi$ i.e. if $\mathcal{T}^{s_0}\subset \mathcal{T}$ then $\phi \upharpoonright \mathcal{T}^{s_0}=\phi_{s_0}$. If now $s_0 \cup \mathcal{T}^{s_0}\subset \mathcal{T}$, then $\phi  \upharpoonright  (s_0\cup \mathcal{T}^{s_0})=\phi_{s_0}\upharpoonright \mathcal{T}^{s_0}$.

           This completes the inductive step and the proof of Theorem $7$.\\
           
           We give a proof now of our second remark.
           
       \begin{proposition} In the contact of Definition $16$, by taking the union of all the strong subtree envelopes of all node-level sets of the family $\mathcal{T}$ and by passing to an infinite strong subtree if necessary, we obtain a uniform family $\mathcal{F}(G)$ of rank less than or equal to the rank of $\mathcal{G}$.
       \end{proposition}
       
       \begin{proof} We give a proof by induction on the rank of $\mathcal{G}$. Suppose that the rank of $\mathcal{G}$ is finite. As we have seen above from the discussion before Lemma $6$, if $(N_1,L_1)$, $(N_2,L_2)\in \mathcal{T}$, then $X_1\in \mathcal{C}^U_{(N_1,L_1)}$ and $X_2\in  \mathcal{C}^U_{(N_2,L_2)}$ are isomorphic and have height equal to $n$. By taking the union of all the strong subtree envelopes of all members of $\mathcal{T}$, we get a family $\mathcal{F}(\mathcal{G})$ of finite strong subtrees of $U$ with height equal to $n$. By applying Corollary $1$ we get a strong subtree $T$ of $U$ so that the second statement of this corollary holds. To see that suppose we get $T\in \mathcal{S}_{\infty}(U)$ such that $\mathcal{S}_n(T)\cap \mathcal{F}(\mathcal{G})=\emptyset$. But $\mathcal{G}\upharpoonright T$ is also a uniform family and the mapping $c$ restricted on that family is canonical. Pick an $X\in \mathcal{G}\upharpoonright T$ and consider $Y\in \mathcal{C}^T_{f(X)}$. Note that $Y\in \mathcal{F}(\mathcal{G})\cap \mathcal{S}_n(T)$, a contradiction. Notice that the elements of any strong subtree envelop have height $n$. Therefore we get a uniform family $\mathcal{F}(G)$ of rank $n$. In fact we get the unique uniform family of rank $n$ on $T$.
       
       Assume now that the rank of $\mathcal{G}$ is $\omega$. By definition $\mathcal{G}(t)$ is of rank $n$, for some $n\in \omega$. By above $\mathcal{F}(\mathcal{G})(t)$ is of rank less than or equal to $n$. Consider the coloring $c': \mathcal{S}_1(U)\to \omega$ defined by $c(t)=n$ if and only if the rank of $\mathcal{F}(\mathcal{G})(t)$ is $n$. By Theorem $5$ we get a strong subtree $T$ of $U$ such that either the coloring is constant and equal to $n_0\in \omega$, one-to-one, or is constant on each level, i.e. $c(t)=c(s)$ if and only if $|t|= |s|$. In the first case the rank of $\mathcal{G}_0$ is $n_0$. In the last two cases the rank of $\mathcal{F}(G)$ is $\omega$.
       
       Suppose now that the rank of $\mathcal{G}$ is $\alpha$, for $\alpha> \omega$ and for all $\beta<\alpha$ our proposition holds. By definition $\mathcal{G}(t)$ is of rank $\beta<\alpha$, so the inductive hypothesis applies and we proceed as in the above paragraph.
       \end{proof}

          Finally we show that our definition of a canonical coloring is the appropriate one.
          
          \begin{proposition} Let $c$ be a canonical coloring of a uniform family $\mathcal{G}$ on $U$ and let $(\mathcal{T}_0,f_0)$, $(\mathcal{T}_1,f_1)$ be two pairs that satisfy the conditions $1$ and $2$ of the Definition $16$. Then there exists $T\in \mathcal{S}_{\infty}(U)$ so that:\\
          \begin{center}$ \mathcal{T}_0\upharpoonright T=\mathcal{T}_1\upharpoonright T$ and $f_0=f_1$ on $\mathcal{G}\upharpoonright T$ \end{center}
         \end{proposition}
         
          \begin{proof} By definition $f_i:\mathcal{G}\to \mathcal{T}_i$ is such that $c(X_0)=c(X_1)$ if and only if $f_i(X_0)=f_i(X_1)$, for $i\in 2$. Let $\mathcal{G}_0$ be the uniform family resulting by taking the union of all the strong subtree envelopes of the node-level sets in $\mathcal{T}_0$ and $\mathcal{G}_1$ the one resulting from $\mathcal{T}_1$. We remind the reader here that both uniform families are assumed to be defined on $U$ instead of one of its infinite strong subtrees.
          By an application of Lemma $8$ on $(\mathcal{G}_0,c_0)$ and $(\mathcal{G}_1,c_1)$, we get $T\in \mathcal{S}_{\infty}(U)$ such that the first statement of the lemma holds. We also notice that both ranks of $\mathcal{G}_0$ and $\mathcal{G}_1$ must be equal by Lemma $11$. By Proposition $2$ we have that $\mathcal{T}_0\upharpoonright T=\mathcal{T}_1\upharpoonright T$.

       We claim that $\mathcal{T}_0\upharpoonright T=\mathcal{T}_1\upharpoonright T$,  implies that $f_0 $ agree with $f_1$ on $\mathcal{G}\upharpoonright T$. To see this suppose that for $X\in \mathcal{G}$ we have that $f_0(X)=(N_0,L_0)\neq f_1(X)=(N_1,L_1)$. Let $X_0\in \mathcal{C}^{T}_{(N_0,L_0)}$ and $X_1 \in \mathcal{C}^{T}_{(N_1,L_1)}$. Then $c_0(X_0)\neq c_0(X_1)$ and $c_1(X_0)\neq c_1(X_1)$. But then $c(X)\neq c(X)$, a contradiction.

                 \end{proof}

      The inductive step of Theorem $8$ is identical with the inductive step of Theorem $7$. Therefore we extended the result of Milliken completing the research along the line of P.Erd\"os and R. Rado.

      Next we mention a possible application of our canonical result.      
      Suppose that $\mathcal{U}$ and $\mathcal{V}$ are ultrafilters on index-sets $X$ and $Y,$ respectively. Let $\mathcal{V}\leq_{RK}\mathcal{U}$ denote the fact that there is a map $F:X\rightarrow Y$ such that
$\mathcal{V}=\{M\subseteq Y: F^{-1}(M)\in\mathcal{U}\}.$
Put $\mathcal{U}\equiv_{RK}\mathcal{V}$ whenever $\mathcal{V}\leq_{RK}\mathcal{U}$ and $\mathcal{U}\leq_{RK}\mathcal{V}.$ This is equivalent to saying that there is a bijection between a set in $\mathcal{U}$ and a set in $\mathcal{V}$ that transfers one ultrafilter into the other. There is a coarser pre-ordering between ultrafilters
that is of a considerable recent interest. This is the \emph{Tukey ordering} which says that $\mathcal{V}\leq_T\mathcal{U}$ if there is a monotone map
$F': \mathcal{U}\rightarrow \mathcal{V}$ whose range generates $\mathcal{V},$ i.e., every element of $\mathcal{V}$ is refined by $F'(M)$ for $M\in \mathcal{U}.$ Recall that a (non-principal) ultrafilter $\mathcal{U}$ on $\mathbb{N}$ is
\emph{selective} if for every map $f:\mathbb{N}\rightarrow\mathbb{N}$ there is $M\in \mathcal{U}$ such that the restriction $f\upharpoonright M$ is either one-to-one or
constant.

In \cite{Ra-To}, S. Todorcevic has used the Theorem $2$ to prove the following result.

\begin{theorem}[\cite{Ra-To}]
Tukey predecessors of a selective ultrafilter on $\mathbb{N}$ are exactly its countable transfinite Fubini powers modulo, of course, the Rudin-Keisler equivalence.
\end{theorem}
  
    In section $4$ we established that $(\mathcal{S}_{\infty}((U_i)_{i\in d}), \subseteq, r)$ forms a topological Ramsey space. It turns out that every topological Ramsey space has the corresponding notion of a selective ultrafilter (see \cite{Mij}). Since we proved the analogue of the Pudl\'ak-R\"odl result for the space of $\mathcal{S}_\infty(U)$ of strong subtrees, we really believe that our Theorem $7$ can be used to characterize the Tukey predecessors of ultrafilters on $\mathcal{S}_1(U)$ that are selective relative to the space $\mathcal{S}_\infty (U)$.

\end{document}